\documentclass[12pt]{amsart}
\usepackage{amsfonts}
\usepackage{amsmath}
\usepackage{amssymb}
\usepackage{graphicx}
\usepackage{enumerate}
\usepackage{multicol}
\usepackage{color}
\usepackage[text={5.35in,8.3in}]{geometry}
\usepackage[arrow, matrix, curve, xdvi, dvips]{xy}

\newcommand{\dR}{\mathbb R}

\def\diag{\mathop{\mathrm{diag}}}
\DeclareMathOperator{\aff}{aff}

\newcommand{\m}{{\mathbf m}}

\newtheorem{theorem}{Theorem}
\newtheorem{corollary}{Corollary}
\newtheorem{lemma}{Lemma}
\newtheorem{definition}{Definition}
\newtheorem{proposition}{Proposition}
\theoremstyle{remark}
\newtheorem{remark}{Remark}
\newtheorem{example}{Example}

\DeclareMathOperator{\Id}{Id }

\DeclareMathOperator{\SO}{SO}
\DeclareMathOperator{\Orth}{O}

\DeclareMathOperator{\Hom}{Hom}
\DeclareMathOperator{\pr}{pr}

\DeclareMathOperator{\tr}{tr}
\DeclareMathOperator{\orth}{\mathfrak{o}}
\DeclareMathOperator{\cal G}{\mathfrak{gl}}

\title[Rolling pseudo-manifolds]{An intrinsic formulation for rolling pseudo-Riemannian manifolds}

\author[I. Markina, F. Silva Leite]{Irina Markina, F\'atima Silva Leite}

\address{Department of Mathematics, University of Bergen, Norway.}
\email{irina.markina@uib.no}

\address{Department of Mathematics and Institute of Systems and Robotics, University of Coimbra, Portugal.}
\email{fleite@mat.uc.pt}

\thanks{The first author was supported by NFR grants \#204726/V30 and 213440/BG. The second author was partially supported by FCT under project PTDC/EEA-CRO/113820/2009.}



\begin{document}

\maketitle


\begin{abstract}
	In the present work we define the rolling of one pseudo-Riemannian manifold over another without slipping and twisting. We compare the definition of the rolling without slipping and twisting of two manifolds isometrically embedded into a pseudo-Euclidean space with the rolling defined only by the intrinsic data, namely by the metric tensors on manifolds. The smooth distribution on the configuration space, encoding the no-slipping and no-twisting kinematic conditions is constructed. Some results concerning the causal character of the rolling curves are also included. Several examples are presented along the paper to illustrate concepts and help to understand the theoretical results.
\end{abstract}


\section{Introduction}

Motions of systems with nonholonomic constraints can be found in the work of great mathematicians as Newton, Euler, Bernoulli and Lagrange. More recently, nonholonomic systems have  attracted  much attention in control
literature due to their numerous applications in physics and  engineering problems. For instance, in a robotic system if the controllable degrees of freedom are less than the total degrees of freedom in the configuration space,  the system is nonholonomic.  Nowadays, the interest in this area is increasing and one can find references to potential applications of nonholonomic systems, for instance, in neurobiology and economics. For a recent survey on non-holonomic systems we refer to~\cite{Bloch,SolYusZeg}.

Nonholonomic constraints can be analyzed from the viewpoint of sub-Riemannian geometry. This is the case when the constraints define a completely non-integrable (or bracket generating) subbundle of the tangent bundle of a Riemannian manifold (see, for instance, \cite{Brockett,LiuSussmann,Montgomery} for work interconnecting sub-Riemannian geometry and control theory).  But, if the manifold is equipped with a pseudo-Riemannian metric (the metric tensor is nondegenerate but not positive definite), we will be in the presence  of problems in sub-pseudo-Riemannian geometry (\cite{ChangMarkinaVasiliev,Groch2,Groch3,KM,KM3}).
The term \emph{semi} is also used in some literature with the same meaning as \emph{pseudo}. Such is the case in \cite{ONeill}, our main reference about semi-Riemannian geometry.

A pair $(M, \widehat{M})$ of $n$-dimensional pseudo-Riemannian manifolds, rolling on each other without slipping and twisting, also form a nonholonomic system posing many theoretical challenges and interesting control problems. To  better understanding the geometry of this motion, one needs tools from sub-pseudo-Rie\-mannian geometry. In a Riemannian context, rolling has been approached from two viewpoints: either regarding the manifolds as subsets of an Euclidean space of higher dimension, or defining rolling intrinsically. The first viewpoint makes sense due to the work of Nash in~\cite{Nash} that guarantees the existence of a global isometric imbedding of any $m$-dimensional Riemannian manifold in some Euclidean space of bigger dimension. The classical definition of rolling, as  given, for instance, in \cite{Sharpe}, corresponds to this extrinsic viewpoint. Based on this general definition, the kinematic equations for rolling particular Riemannian manifolds have been derived for instance in \cite{HKS,HS,Zimm}. An intrinsic formulation of rolling is the approach taken in \cite{AS} and \cite{BH} for $2$-surfaces and generalized in \cite{ChKok,GrMol,GGLM,GGLM2} for arbitrary Riemannian manifolds of any dimension. We want to mention, that in~\cite{Grong} a rolling without slipping or twisting of $n$-dimensional manifolds endowed with a connection, not necessarily compatible with any kind of metric, were defined and in~\cite{ChKok} even more general constructions for tensors bundles were made. Our paper develops the ideas of~\cite{GGLM2}, explaining the relation between extrinsic and intrinsic approach, and provides numerous examples illustrating main ideas and showing new features of the presence of the pseudo-Riemannian metrics in contrast to the Riemannian ones on rolling manifolds.

When the manifolds $M$ and $\widehat{M}$ are both isometrically embedded in some bigger pseudo-Riemannian manifold $\overline{M}$, one can develop an extrinsic formulation of rolling, as a rigid motion inside $\overline{M}$, subject to no-slip and no-twist constraints. This si\-tuation has been explored for some particular cases where $M$ is a Lorentzian sphere (\cite{KLF}), $M$ is a pseudo-hyperbolic space (\cite{ML}), and  $M$ is a pseudo-orthogonal group (\cite{CL}). In all these cases, $\widehat{M}$ has been chosen to be the affine tangent space of $M$ at a point $p_0$. It turns out that any pseudo-Riemannian manifold has a global isometric embedding into a pseudo-Euclidean space (\cite{Clarke}). So, as in the Riemannian situation, both the extrinsic and the intrinsic approaches make sense. As far as we know, the rigorous intrinsic viewpoint of rolling has not been developed in the pseudo-Riemannian case.

The structure of the present paper is the following. After introducing the basic notations in Section~\ref{sec_basic} we present the definition of the extrinsic rolling in Section~\ref{sec:Def}.
We start with the generalization of the classical  defi\-ni\-tion of rolling given in \cite{Sharpe}, with some convenient adaptations as done in~\cite{GGLM2}. At this stage we assume that $M$ and $\widehat{M}$ are both isometrically embedded in $\dR^n_\nu$, the pseudo-Euclidean space of dimension $n$ and index $\nu$. We then proceed with the intrinsic definition of rolling in Section~\ref{introll}, where we compare the intrinsic component of the rolling map that depends only on metric data with the extrinsic part, that involves the information about concretely chosen isometric embedding. In Section~\ref{sec:dist} we present the smooth distribution on the configuration space caring kinematic restrictions of no slipping and no twisting. The causal character of the rolling map is studied in Section~\ref{Causal character}, where we give some conditions under which the causal character of a rolling curve is preserved. The last Section~\ref{extended} reveals the idea of inclusion of the configuration space of the rolling problem as a smooth sub-bundle to a vector bundle. Notes that the configuration space is defined as a smooth fiber bundle with typical fiber isomorphic to a group of pseudo-Euclidean rotations.


\section{Basic facts about pseudo-Riemannian\\geometry}\label{sec_basic}

We start with the basic background about pseudo-Riemannian geometry that will appear throughout the paper. For more details, we refer to O'Neill~\cite{ONeill}. A pseudo-Riemannian manifold is a smooth manifold $\overline{M}$ furnished with a metric tensor $\overline{g}$ (a symmetric
nondegenerate ($0,2$) tensor field of constant index). The common value $\nu$ of the index $\overline{g}_x$
at each point $x$ on a pseudo-Riemannian manifold $\overline{M}$ is called the index of $\overline{M}$ and $0\leq \nu \leq
\dim{(\overline{M})}$. If $\nu =0$, each $\overline{g}_x$ is then a (positive definite) inner product on $T_x\overline{M}$
 and $\overline{M}$ is a Riemannian manifold. If $\nu=1$ and $\dim{(\overline{M})}\geq 2$, $\overline{M}$ is called a Lorentz manifold.

 If $(\overline{M},\overline{g} )$ is a pseudo-Riemannian manifold and $v\in T_x\overline{M}$, then
$v$ is \emph{spacelike} if $\overline{g}(v,v)>0$ or $v=0$; $v$ is \emph{timelike} if $\overline{g}(v,v)<0$;
$v$ is \emph{null} if $\overline{g}(v,v)=0$ and $v\neq 0$.
Since $\overline{g}(v,v)$ may be negative, the norm $|v|$ of a vector is defined to be $|v|:=|\overline{g}(v,v)|^{1/2}$. A unit vector $v$ is a vector with  norm $1$, that is $\overline{g}(v,v)=\pm 1$. As usual, a set of mutually orthogonal unit vectors is said to be \emph{orthonormal}. It is known that always there is an orthonormal basis, such that first $\nu$ vectors are unite timelike and the rest $n-\nu$ are unite spacelike orthogonal vectors~\cite{ONeill}.

   Let $M$ be a submanifold of a pseudo-Riemannian manifold $(\overline{M},\overline{g})$  and $\imath: M\hookrightarrow
\overline{M}$  the inclusion map. Then $M$ is a pseudo-Riemannian submanifold of $\overline{M}$ if the pullback metric
$g=\imath^*(\overline{g})$ is a metric tensor on $M$. If $M$ is equipped with the induced metric $g$, then $\imath$ is an
isometric embedding. In subsequent sections, we use  $\langle  \cdot,\cdot \rangle$ as an alternative notation for $g$.

    Let $M$ be a pseudo-Riemannian submanifold of $\overline{M}$ (write $M\subset \overline{M}$), and $x\in M$. Each tangent
space $T_xM$ is, by definition, a nondegenerate subspace of $T_x\overline{M}$. Consequently, $T_x\overline{M}$ decomposes as
 a direct sum
 \begin{equation}\label{splitting}
 T_x\overline{M}=T_xM\oplus T^{\perp}_xM, \quad \forall\ x \in M,
 \end{equation}
 and  $T^{\perp}_xM$ is also nondegenerate. Vectors in $T^{\perp}_xM$ are said to be normal to $M$, while those
in $T_xM$ are, of course, tangent to $M$. Similarly, a vector field $Z$ on $\overline{M}$ is normal (respectively tangent)
to $M$ provided each value $Z_x$, for $x\in M$  belongs to  $T^{\perp}_xM$ (respectively $T_xM$).

If $X,Y$ are vector fields on $M$, we can extend them to $\overline{M}$, denoting as $\overline X$, $\overline Y$, apply the ambient Levi Civita connection $\overline{\nabla}$ with respect to $\overline{g}$ and then decompose at points of $M$ to get
\begin{equation} \label{normaldecomp}
\overline{\nabla}_{\overline X}\overline Y=\Big(\overline{\nabla}_{\overline X}\overline Y\Big)^{\top}+ \big(\overline\nabla_{\overline X}\overline Y\big)^{\perp}=\nabla_XY +\big(\overline\nabla_{\overline X}\overline Y\big)^{\perp},
\end{equation}
where $\nabla$ is a Levi-Civita connection with respect to the induced metric on $M$ and the last term, given by the orthogonal projection to $T^{\perp}M$, measures the difference between the intrinsic connection $\nabla$ on $M$ and the ambient connection $\overline{\nabla}$ on $\overline M$.

The analogous considerations can be done for normal vector fields on $M$.
If $X$ is a tangent vector field and $Z$ is a normal vector field to $M$, we have
\begin{equation} \label{normaldecomp1}
\overline{\nabla}_{\overline X}{\overline Z}=(\overline{\nabla}_{\overline X}{\overline Z} )^{\top} +\nabla^{\perp}_XZ,
\end{equation}
where $\nabla^{\perp}$ is the \emph{normal connection} of $M\subset \overline{M}$, that is the function $\nabla^{\perp}$ that, to each pair $(X,Z)$ of smooth vector fields,  $X$  tangent to $M$ and $Z$ normal to $M$, assigns a vector field $\nabla^{\perp}_XZ$ normal to $M$.

If $t\mapsto \gamma(t)$ is a curve in $M$, $V$ is a smooth vector field tangent to $M$ along $\gamma$, and $W$ is a smooth vector field normal to $M$ along $\gamma$, then the formulas~\eqref{normaldecomp} and~\eqref{normaldecomp1} have their analogous in terms of covariant derivatives along $\gamma$:
  \begin{equation} \label{gaussformulacurves}
\frac{\overline{D}}{dt}\overline V=\frac{D}{dt}V+\Big(\frac{{\overline D} }{dt}\overline V\Big)^{\perp},\qquad
\frac{\overline{D}}{dt}\overline W=\Big(\frac{\overline D}{dt}\overline W\Big)^{\top}+\frac{D^{\perp}}{dt}W,
\end{equation}
where $\frac{\overline{D}}{dt}$ ($\frac{D}{dt}$) denote extrinsic (intrinsic) covariant derivative along $\gamma$, $\frac{D^{\perp}}{dt}$ is the normal covariant derivative along $\gamma$, and $\overline V$, $\overline W$ are extensions of $V$ and $W$ in a neighborhood of $\gamma$ considered as a curve in $\overline M$, see~\cite{Lee,ONeill}.

All curves are assumed to be absolutely continuos. A tangent vector field $V$ along a curve $\gamma$ is said to be a \emph{tangent parallel vector field} along $\gamma$ if $\frac{DV}{dt}\equiv 0$ for almost all $t$. Analogously, a normal vector field $Z$ along $\gamma$ is said to be a \emph{normal parallel vector field} along $\gamma$ if $\frac{D^{\perp}Z}{dt}\equiv 0$ for almost all $t$.

 The following  holds, both for tangent and for normal parallel vector fields along curves in $M$.
 \begin{lemma}~\cite{Lee}
 Let $[a,b]\ni t\mapsto \gamma(t)$ be an absolutely continuous curve in~$M\subset \overline{M}$.
 \begin{itemize}
 \item[(1)] If $Y_0\in T_{\gamma (a)}M$, then there is a unique tangent parallel vector field $Y$ along $\gamma$ such that $Y(a)=Y_0$.
\item[(2)] If $Z_0\in T^{\perp}_{\gamma (a)}M$, then there is a unique normal  vector field $Z$ along $\gamma$ such that $Z(a)=Z_0$.
 \end{itemize}
 \end{lemma}

 With the notations above, the map
  \begin{equation} \label{transport}
\begin{array}{cccc}
P_a^b(\gamma)\colon &T_{\gamma (a)}M &\rightarrow &T_{\gamma (b)}M
\smallskip
\\
&Y(a)&\mapsto &Y(b)
\end{array}
\end{equation}
is called the \emph{tangent parallel translation} of $Y_0$ along $\gamma$, from the point $p=\gamma (a)$ to the point $q=\gamma (b)$.
 Similarly,
 \begin{equation} \label{transport1}
\begin{array}{cccc}
P_a^b(\gamma)\colon &(T_{\gamma (a)}M )^{\perp}&\rightarrow &(T_{\gamma (b)}M)^{\perp}
\smallskip
\\
&Z(a)&\mapsto &Z(b)
\end{array}
\end{equation}
is called the \emph{normal parallel translation} of $Z_0$ along $\gamma$.

Both,  tangent and the normal parallel translations are linear isometries. Consequently, tangent (respectively normal) parallel translation of a tangent (respectively normal) frame  gives a tangent (respectively normal) parallel frame field along $\gamma$. An absolutely continuous curve $t\mapsto \gamma (t) $ in $M$ is a {\it geodesic} if its velocity vector field is parallel along $\gamma$, i.e., $\nabla_{\dot{\gamma}}\dot\gamma(t)=0$ for almost all $t$.  In pseudo-Riemannian geometry  there are three types of geodesics, determined by the causal character of the initial velocity vector. More specifically, $\gamma$ is a spacelike geodesic (respectively, timelike or null) if $\dot{\gamma}(0)$ is spacelike (respectively, timelike or null).
 The theory of pseudo-Riemannian geometry guarantees that a geodesic starting at $p_0$ with initial velocity $V_0$ is locally unique.


\section{Rolling submanifolds of pseudo-Euclidean\\spaces}\label{sec:Def}


The present section is devoted to the geometrical formulation of the rolling of a pseudo-Riemannian manifold $M$ over another $\widehat{M}$, while both are embedded into a pseudo-Euclidean space. Since any pseudo-Riemannian manifold may be globally isometrically embedded in a pseudo-Euclidean space, see~\cite{Clarke}, we assume that the manifolds $M$ and $\widehat{M}$ are connected, have the same dimension $m$, index $\mu$, and are both embedded in some $\dR_\nu^n$, which is the vector space $\dR ^n$ endowed with the pseudo-Riemannian metric induced by matrix $J=\diag (-I_{\nu}, I_{n-\nu})$. That is, for any vectors $x,y \in \dR ^n$, $\langle x,y\rangle_J=x^{\mathbf t}Jy$, where $x^{\mathbf t}$ is $x$ transposed. We identify the abstract manifolds $M$ and $\widehat{M}$ with their images under this  embedding. A rolling motion of $M$ over $\widehat{M}$ is a rigid motion inside  $\overline{M}=\dR_\nu^n$ and as such it is described by the action of the group of isometries of $\dR_\nu^n$, which is known (see, for instance, \cite[p. 240]{ONeill}) to be $\overline{\mathbb G}=\dR_{\nu} ^n\rtimes \Orth_{\nu} (n)$, where $\Orth_{\nu} (n)$ is the pseudo-orthogonal group
$$\Orth_{\nu} (n)=\{ X\in
GL(n)|\,\,X^{\mathbf t}JX=J\}.
$$
 The group $\overline{\mathbb G}$ is known as the pseudo-Euclidean group. It follows from the definition that all matrices in $\Orth_{\nu} (n)$ have determinant equal to $\pm 1$. Elements  in $\overline{\mathbb G}$ can be represented by pairs $(s,A)$, multiplication is defined as $(s_1,A_1)(s_2,A_2)=(s_1+A_1s_2,A_1A_2)$ and $(s,A)^{-1}=(-A^{-1}s,A^{-1})$. The action of $\overline{\mathbb G}$ on $\dR_\nu^n$ is defined by $(s,A)x=s+Ax$, for any vector $x\in \dR_\nu^n$.
 In the case $\nu=0$, we have the Riemannian situation and the group of isometries is the Euclidean group of rigid motions in $\dR ^n$. Let us concentrate for a while on the group $\Orth_{\nu} (n)$, for $\nu\neq 0$ and an arbitrary~$n$. A matrix  $A\in\Orth_{\nu} (n)$ can be written in block form as
\begin{equation*}\label{blocos}
A=\left[\begin{array}{c|c}
                     A_T&  B
                     \\
                     \hline
                     C  & A_S
                   \end{array}
\right],
\end{equation*}
where $A_T$ and $A_S$ are invertible matrices of order $\nu $  and  $n-\nu $ respectively. An element $A \in \Orth_{\nu} (n)$ preserves (reverses) time orientation provided that $\det(A_T)>0$ ($<0$), and preserves (reverses) space orientation provided that  $\det(A_S)>0$ ($<0$). $\Orth_{\nu} (n)$ can then be split into four disjoint sets $\Orth_{\nu} ^{++}(n)$, $\Orth_{\nu} ^{+-}(n)$, $\Orth_{\nu} ^{-+}(n)$, and $\Orth_{\nu} ^{--}(n)$, indexed by the signs of the determinants of $A_T$ and $A_S$, in this order. The following three disconnected subgroups of $\Orth_{\nu} (n)$ play an important role in orientability of pseudo-Riemannian manifolds:
\begin{equation}
\Orth_{\nu} ^{++}(n)\cup \Orth_{\nu} ^{--}(n),\quad \Orth_{\nu} ^{++}(n)\cup \Orth_{\nu} ^{+-}(n),\quad \Orth_{\nu} ^{++}(n)\cup \Orth_{\nu} ^{-+}(n).
\end{equation}
According to \cite{ONeill}, if we denote these groups by a common $G$,  there are three types of $G$\emph{-orientation}:
\begin{equation}\label{G-orientation}
\begin{array}{l}
 \mbox{orientation if $G=\Orth_{\nu} ^{++}(n)\cup \Orth_{\nu} ^{--}(n)$;}\\
 \mbox{time-orientation if $G=\Orth_{\nu} ^{++}(n)\cup \Orth_{\nu} ^{+-}(n)$;}\\
\mbox{space-orientation if $G=\Orth_{\nu} ^{++}(n)\cup \Orth_{\nu} ^{-+}(n)$.}
\end{array}
\end{equation}
The connected component containing the identity is $\Orth_{\nu} ^{++}(n)$ preserves time  orientation, space orientation, and the orientation of the manifold.
If $V$ is a vector space and $e=\{e_1, \cdots , e_n\}$ and $f=\{f_1, \cdots , f_n\}$ are two orthonormal bases for $V$, the relation $f_j=\sum_ia_{ij}e_i, \quad 1\leq j\leq n$, defines a matrix $A=(a_{ij})\in \Orth_{\nu} (n)$. The bases $e$ and $f$ are $G$\emph{-equivalent} if $A\in G\subset\Orth_{\nu}(n)$. For each $G$ there are two possible $G$-orientations of $V$. A $G$-orientation of a pseudo-Riemannian manifold is a function $\lambda_M$ that assigns to each $x\in M$ a smooth $G$-orientation of $T_xM$, in the sense that there is a coordinate system whose induced local $G$-orientation agrees with $\lambda_M$  on some neighborhood of $x\in M$. $M$ is said to be $G$-orientable provided it admits a $G$-orientation. More details about the orientation of pseudo-Riemannian manifolds can be found in~\cite{ONeill}.

The Lie algebra of  $\Orth_{\nu} (n)$, equipped with the Lie bracket defined by the commutator,   is the set
$$\orth_{\nu} (n)=\{ \mathcal A\in
{\cal G}(n)|\,\, \mathcal A^{\mathbf t}J=-J\mathcal A\}.$$
We are now ready to generalize the classical definition of a rolling motion, as  given in~\cite{GGLM2}, which is an adaptation of the Euclidean definition in~\cite{Sharpe}. In  the present case, the special Euclidean group is replaced by the pseudo-Euclidean group, orthogonality is understood with respect to the pseudo-Riema\-nni\-an metric, and the orientability condition varies according to the choice of one of the three subgroups of $G$. So, the following definition is indexed by the choice of one of the subgroups $G$,  in~\eqref{G-orientation} above, further denoted by $G_{\nu}(n)$.
Recall that  the pseudo-Riemannian manifolds $M$ and $\widehat{M}$ are assumed to have the same dimension $m$ and the same index $\mu$ (not necessarily the same as the embedding space) and both of them are $G$-oriented.

\begin{definition} \label{rolling-classical}
A $G$-rolling of $M$ on $\widehat{M}$ without slipping or twisting is an absolutely continuous curve $(x,g):~[0,\tau] \to M \times \dR_{\nu} ^n\rtimes G_{\nu}(n)$ satisfying the following conditions:
\begin{itemize}
\item[(i)] $\widehat{x}(t) := g(t)\, x(t) \in \widehat{M}$ for almost every $t$,
\item[(ii)] $d_{x(t)}g(t)\, T_{x(t)}M = T_{\widehat{x}(t)} \widehat{M}$ for almost every $t$,
\item[(iii)] $d_{x(t)} g(t)|_{T_{x(t)}M}:T_{x(t)} M \to T_{\widehat{x}(t)} \widehat{M}$ preserves $G$-orientation.
\item[(iv)] No slip condition: $\dot{\widehat{x}}(t)= d_{x(t)}g(t)\, \dot{x}(t),$ for almost every $t$.
\item[(v)] No twist condition $($tangential part$)$:
$$d_{x(t)}g(t)\, \frac{D}{dt}\, Z(t) = \frac{D}{dt}\, d_{x(t)}g(t)\, Z(t),$$
for any tangent vector field $Z(t)$ along $x(t)$ and almost every $t$.
\item[(vi)] No twist condition $($normal part$)$:
$$d_{x(t)}g(t)\,\frac{D^{\perp}}{dt}\, \Psi(t) = \frac{D^{\perp}}{dt}\, d_{x(t)}g(t)\, \Psi(t),$$
for any normal vector field $\Psi(t)$ along $x(t)$ and almost every $t$.
\end{itemize}
\end{definition}

 The curve $x$ is called the \emph{rolling curve}, while $\widehat{x}$ is called the {\it development} of $x$ on $\widehat{M}$. Note that, due to the splitting (\ref{splitting}), the condition  (ii) implies that $d_{x(t)}g(t)\, T^{\perp}_{x(t)}M = T^{\perp}_{\widehat{x}(t)} \widehat{M}$.

The no twist conditions (v) and (vi) have an equivalent formulation involving the notion of parallel vector fields.
\begin{proposition}\label{new-twist}
Assume that condition  {\rm(ii)} holds. Then, conditions {\rm(v)} and {\rm(vi)} are respectively equivalent to:
 \begin{itemize}
\item[(v')] A vector field $Z(t)$ is tangent parallel along $x(t)$ if and only if $d_{x(t)}g(t)\, Z(t)$ is   tangent parallel along $\widehat{x}(t)$;
    \item[(vi')] A vector field $\Psi(t)$    is normal parallel along $x(t)$ if and only if  $d_{x(t)}g(t)\, \Psi(t)$ is   normal parallel along $\widehat{x}(t)$.
 \end{itemize}
 \end{proposition}
\begin{proof}
 Assume that (ii) holds. That is, $d_{x(t)}g(t)$ is a linear isomorphism between $T_{x(t)}M$ and $T_{\widehat{x}(t)}\widehat{M}$. We prove the equivalence between (v) and (v'). The proof of the equivalence of (vi)  and (vi') can be done similarly.

  First, assume that (v) also holds. Then, it is obvious that  $\frac{D}{dt}\, Z(t)=0$ if and only if $\frac{D}{dt}(d_{x(t)}g(t)Z(t))=0$. So, (v) $\Rightarrow $ (v').

 To prove that (v') $\Rightarrow $ (v), let $Z(t)$ be any tangent vector field along $x(t)$ and  $\{e_1(t),\cdots e_m(t)\}$ a parallel tangent frame field along $x(t)$, so that $$Z(t)=\sum_iz_i(t) e_i(t)\quad  \mbox{and} \quad \frac{D}{dt}\,Z(t)=\sum_i\dot{z}_i(t) e_i(t).$$ If $\widehat{e}_i(t):=d_{x(t)}g(t)e_i(t)$, then we can guarantee by assumption that the frame \linebreak $\{\widehat{e}_1(t),\cdots \widehat{e}_m(t)\}$ is also parallel along $\widehat{x}(t)$. So
 $$
 d_{x(t)}g(t)(\frac{D}{dt} \, Z(t))=\sum_i\dot{z}_i(t)d_{x(t)}g(t)e_i(t)=\sum_i\dot{z}_i(t)\widehat{e}_i(t),
 $$
and
  $$\frac{D}{dt} (d_{x(t)}g(t)\, Z(t))=\frac{D}{dt}\, \Big(\sum_iz_i(t) \widehat{e}_i(t)\Big)=\sum_i\dot{z}_i(t)\widehat{e}_i(t).
  $$
  Consequently, $$
 d_{x(t)}g(t)(\frac{D}{dt} \, Z(t))=\frac{D}{dt} (d_{x(t)}g(t)\, Z(t)),
 $$
 proving that (v') $\Rightarrow $ (v).
\end{proof}

 As a consequence of these equivalences, one can replace in Definition \ref{rolling-classical} conditions (v)-(vi) by conditions (v')-(vi').

We also note that for manifolds of dimension one (v') is automatically satisfied, while for embeddings of codimension one (vi') holds automatically.

From now on, whenever we write ``a rolling'' we mean  ``a $G$-rolling without slipping or twisting''.

The following example  of rolling an $m$-dimensional Lorentz\-ian sphere on the affine tangent space at a point $x_0$, both embedded in the pseudo-Euclidean space $\dR_1^{m+1}$, is taken from \cite{KLF}. For the sake of completeness, we work the details here using Definition~\ref{rolling-classical}. We also use this example as a benchmark for several properties that will be proved in a more general context in later sections.


\subsection{A benchmark example - the Lorentzian sphere $S_1^{m}$ rolling over the affine tangent space} \label{Lorentzian sphere}


Let $\overline{M}=\dR_1^{m+1}$,  $M=S_1^{m}=\big\{x\in\dR_1^{m+1}\colon \langle x,x\rangle_J=1\big\}$, with  $J=\diag(-I_1,I_m)$, and the affine tangent space $\widehat{M}=T_{x_0}^{\aff}S_1^{m}$, for some $x_0\in S_1^{m}$. The Lie algebra of the group $G_{1}(m+1)$ is denoted by $\mathfrak g_{1}(m+1)$. The following are easy to check or they are consequence of definitions.
\begin{enumerate}
\item  $T_{x_0}S_1^{m}=\big\{v\in\dR_1^{m+1}\colon v=\Omega x_0,\, \Omega\in\mathfrak g_{1}(m+1)\big\};$
  \item $T^{\aff}_{x_0}S_1^{m}=\big\{v\in\dR_1^{m+1}\colon v=x_0+\Omega x_0,\;\;\Omega\in\mathfrak g_{1}(m+1)\big\} ;$
   \item $T_{x_0}^{\perp}S_1^{m}=\mbox{span}\{ x_0\};$
  \item ${\rm Ad}_R(\Omega)=R\,\Omega R^{-1} \in \mathfrak g_{1}(m+1)$, for every $R \in G_1(m+1)$ and $\Omega \in \mathfrak g_{1}(m+1)$;
  \item $\langle .\,,.\rangle_J$ is $G_1(m+1)$-invariant.
  \item The Lie group $G_1(m+1)$ acts transitively on $S_1^{m}$, consequently any curve $t\mapsto x(t)$ satisfying $x(0)=x_0$ is of the form $x(t)=R(t)x_0$, for some $R(t)\in G_1(m+1)$ satisfying the conditions $R(0)x_0=x_0$. If, in particular, $R(0)=I_{m+1}$, then $R(t)$ is a curve in $O^{++}_1(m+1)$.
\end{enumerate}


\subsubsection{\textbf{Kinematic equations for rolling the Lorentzian sphere}}


Let $t\mapsto u(t)$ be an absolutely continuous function satisfying $\langle u(t),x_0\rangle_J=0$ and $t\mapsto (s(t),R(t))\in \overline{\mathbb G}=\dR_{1} ^{m+1}\rtimes G_1(m+1)$ a curve in $\overline{\mathbb G}$, satisfying $(s(0),R(0))=(0,I_{m+1})$, with velocity vector (whenever defined) given by
\begin{equation}\label{Kinematic-Lorentz}
  \begin{array}{lcl}
\dot{s}(t) &=& u(t),\\
\dot{R} (t)&=&R(t)\left(u(t)x_{0}^{\mathbf t}-x_{0}u^{\mathbf t}(t)\right)\, J.
  \end{array}
\end{equation}
We prove that $t\mapsto (x(t),g(t))\in S_1^{m}\times \overline{\mathbb G} $, where $ x(t)=R(t)x_0$ and $g(t)=(s(t), R^{-1}(t))$, $R(t)\in O^{++}_1(m+1)$ is a rolling of $S_1^{m}$ over $T^{\aff}_{x_0}S_1^{m}$, by showing that the first five conditions in Definition~\ref{rolling-classical} hold. Equations~\eqref{Kinematic-Lorentz} are called the~\emph{kinematic equations} for rolling the Lorentzian sphere over the affine tangent space at the point $x_0$. Condition (vi) is automatically satisfied since this is a co-dimension one case.

\noindent \emph{Proof of {\rm(i)}.} We have $\widehat{x}(t):=g(t)x(t)=s(t)+R^{-1}(t)x(t)=s(t)+x_0$.
Since $\dot{s}(t) = u(t) \in T_{x_0}S_1^{m}$ and $s(0)=0$, then $s(t) \in T_{x_0}S_1^{m}$ and $\widehat{x}(t)=s(t)+x_0 \in T^{\aff}_{x_0}S_1^{m}$.

\noindent \emph{Proof of {\rm(ii)}.}   Elements in $T_{x(t)}S_1^{m}$ are of the form $\Omega (t) x(t)$, with $\Omega (t) \in \mathfrak g_1(m+1)$. So,
\begin{equation*}
\begin{array}{lcl}d_{x(t)}g(t)(\Omega (t) x(t))&=&R^{-1}(t) \Omega (t) x(t)\\
&=&\underbrace{R^{-1}(t) \Omega (t) R(t)}_{\in \mathfrak g_1(m+1)}x_0 \in T_{x_0}S_1^{m}.
\end{array}\end{equation*}
Since $T_{\widehat{x}(t)}(T^{\aff}_{x_0}S_1^{m})$ is identified with $T_{x_0}S_1^{m}$,  the result follows.

\noindent \emph{Proof of {\rm(iii)}.}  The map $d_{x(t)}g(t)=R^{-1}(t)\colon T_{x(t)}S_1^{m}\to T_{\widehat{x}(t)}\widehat{M}$ is linear for all $t$ whenever it is defined. Since $R(t)$ is a continuous curve in $G_1(m+1)$ and $R(0)=I_{m+1}$, $R(t)$ and its inverse must remain in the connected component containing the identity of $G_1(m+1)$, which is $\Orth^{++}_1(m+1)$, so keeping the sign of the determinant for all $t$ that guarantees that $d_{x(t)}g(t)$ is orientation preserving. 

\noindent \emph{Proof of {\rm(iv)}.} We now have to use constraints on velocity given by~\eqref{Kinematic-Lorentz}.
  $$
  \begin{array}{lcl}
  d_{x(t)}g(t)\dot{x}(t)&=&R^{-1}(t)\dot{x}(t)=R^{-1}(t)\dot{R}(t)x_0\\
  &=& (u(t)x_{0}^{\top}-x_{0}u^{\top}(t))\, J \, x_0\\
  &=&\langle x_0,x_0\rangle_J u(t)- \langle u(t),x_0
\rangle_J x_0  =u(t).
  \end{array}
  $$
On the other hand, $\widehat{x}(t)=s(t)+x_0$ from the proof of (i). So
  $
  \dot{\widehat{x}}(t)=\dot{s}(t)=u(t)
  $,
and the identity in (iv) holds.

 \noindent \emph{Proof of {\rm(v)}.}  The covariant derivative of a tangent vector field $Z(t)$ along $x(t)$ is a tangent vector field  along $x(t)$ that results from orthogonal projection of the extrinsic derivative $\dot{Z}(t)$  on the tangent space $T_{x(t)}S_1^{m}$. That is,
      \small{$\frac{D}{dt}\, Z(t)=\dot{Z}(t)-\langle \dot{Z}(t),x(t)
\rangle_J\, x(t)$.} So,
 \small{ \begin{align*}
  d_{x(t)}g(t)\frac{D}{dt}\, Z(t)&= R^{-1}(t)\big(\dot{Z}(t)-\langle \dot{Z}(t),x(t)\rangle_J\, x(t)\big)
 \\
 & =R^{-1}(t)\dot{Z}(t)-\langle R^{-1}(t)\dot{Z}(t),R^{-1}(t)x(t)\rangle_J\, R^{-1}(t)x(t)
 \\
  &=R^{-1}(t)\dot{Z}(t)-\langle R^{-1}(t)\dot{Z}(t),x_0\rangle_J\, x_0.
  \end{align*}}
On the other hand, since $$d_{x(t)}g(t)\, Z(t)= R^{-1}(t)Z(t)\in T_{\widehat{x}(t)}\widehat{M}\cong T_{x_0}S_1^{m},$$ we have
  \begin{align*}
  \frac{D}{dt}d_{x(t)} &g(t)\, Z(t)= \frac{D}{dt}R^{-1}(t)Z(t)
  \\
  &=\dot{R^{-1}}(t)Z(t)+R^{-1}(t)\dot{Z}(t)-\langle \dot{R^{-1}}(t)Z(t)+R^{-1}(t)\dot{Z}(t),x_0\rangle_Jx_0
  \\
 & =R^{-1}(t)\dot{Z}(t)-\langle R^{-1}(t)\dot{Z}(t),x_0\rangle_J\, x_0
  \\
  &+\dot{R^{-1}}(t)Z(t)-\langle \dot{R^{-1}}(t)Z(t),x_0\rangle_Jx_0.
  \end{align*}
So, in order to prove  (v) we have to show that the sum of the last two terms in the previous expression equals $0$. For this, take into consideration that $Z(t)=\Omega(t)x(t)= \Omega(t)R(t)x_0$, for some $\Omega(t)\in \mathfrak g_1(m+1)$, and $\dot{R^{-1}}=-R^{-1}\dot{R}R^{-1}$, to obtain
  \begin{align*}
 \dot{R^{-1}}(t)&Z(t) =-R^{-1}(t)(u(t)x_{0}^{\mathbf t}-x_{0}u^{\mathbf t}(t)) JR^{-1}(t)\Omega(t)R(t)x_0
 \\
& =-\underbrace{\langle x_0,R^{-1}(t)\Omega(t)R(t)x_0\rangle_J}_{=0}\, u(t)+\langle u(t),R^{-1}(t)\Omega(t)R(t)x_0\rangle_J\, x_0
\\
& =\langle u(t),R^{-1}(t)\Omega(t)R(t)x_0\rangle_J\, x_0,
 \end{align*}
and, consequently,
  \begin{align*}
 \langle \dot{R^{-1}}(t)Z(t),x_0\rangle_Jx_0 & = \langle u(t),R^{-1}(t) \Omega(t)R(t)x_0\rangle_J\, \langle x_0,x_0\rangle_J\, x_0
 \\
 &= \langle u(t),R^{-1}(t)\Omega(t)R(t)x_0\rangle_J\, x_0=\dot{R^{-1}}(t)Z(t),
 \end{align*}
completing the proof of (v).


 \subsubsection{\textbf{Rolling versus parallel translation}} \label{Rolling versus parallel transport}

 We show that parallel translation of a given vector $Y_0$ along a curve in $ S_1^{m}$ can be realized by using the rolling along that curve.

More precisely, we show that if $x(t)=R(t)x_0$ is a rolling curve satisfying the initial condition $x(0)=x_0$, rolling map $g(t)=(s(t), R^{-1}(t))$ with $g(0)=\big(0,I_{m+1}\big)$, and $Y_0 \in T_{x_0}S_1^{m}$, then
 $Y(t)=R(t)Y_0 $ is the unique tangent parallel vector field along $x(t)$ satisfying $Y(0)=Y_0$. Similarly, if $\Psi_0 \in T^{\perp}_{x_0}S_1^{m}$, then
 $\Psi(t)=R(t)\Psi_0 $ is the unique normal parallel vector field along the curve $x(t)$ satisfying~$\Psi(0)=\Psi_0$.

 To prove the first statement, we notice that if $Y_0 \in T_{x_0}S_1^{m}$, then
 $$\langle Y(t),x(t)\rangle_J=\langle R(t)Y_0,R(t)x_0\rangle_J= \langle Y_0,x_0\rangle_J=0,\quad\Longrightarrow\quad Y(t)\in T_{x(t)}S_1^{m}.
 $$
 We now have to show that $\frac{DY}{dt}=0$, where, in this case, $\frac{DY(t)}{dt} =\dot{Y}(t)-\langle \dot{Y}(t), x(t)\rangle_Jx(t)$. Using the second kinematic equation in~\eqref{Kinematic-Lorentz} and the conditions $\langle Y_0,x_0\rangle_J=0$, $\langle x_0,x_0\rangle_J=1$, we may conclude after simplifications that
 $$
 \dot{Y}(t)=\dot{R}(t)Y_0=-\langle u(t),Y_0\rangle_J\, R(t)x_0;
 $$
and
$$
\langle \dot{Y}(t), x(t)\rangle_Jx(t)=-\langle u(t),Y_0\rangle_J\, R(t)x_0.
$$
So,  $\frac{DY}{dt}=0$, i.e. $Y(t)=R(t)Y_0$ is the unique parallel vector field along $x(t)$ satisfying $Y(0)=Y_0$.

For the second statement, notice that if $\Psi_0 \in T^{\perp}_{x_0}S_1^{m}$, then $\Psi_0=kx_0$, for some $k\in \mathbb{R}$, and consequently $\Psi(t)=R(t)\Psi_0=kx(t)\in T^{\perp}_{x(t)}S_1^{m}$. So, in this case, using similar arguments and the fact that $\langle u(t),x_0\rangle_J=0$, one has
$$
\langle \dot{\Psi}(t), x(t)\rangle_J=\langle \dot{R}(t)\Psi_0, x(t)\rangle_J=\langle \dot{R}(t)\Psi_0, R(t)x_0\rangle_J=k\langle u(t),x_0\rangle_J=0.
$$
Consequently, $\frac{D^{\perp}\Psi}{dt} =\langle \dot{\Psi}(t), x(t)\rangle_Jx(t)=0$, for almost all $t$, that is, $\Psi$ is the unique normal parallel vector field along $x(t)$ satisfying $\Psi(0)=\Psi_0$.


\subsubsection{\textbf{Causality}}\label{causality}


For the Lorentzian sphere, it can easily be shown that the rolling curve and its development have the same causal character. Indeed, using results from the previous subsection, namely $x(t)=R(t)x_0$, $\widehat{x}(t)=s(t)+x_0$, $\langle x_0,x_0\rangle_J=1$, $\langle u(t),x_0\rangle_J=0$, and the kinematic equations~\eqref{Kinematic-Lorentz}, we can write
$$
\begin{array}{lcl}
\langle \dot{\widehat{x}}(t),\dot{\widehat{x}}(t)\rangle_J&=&\langle \dot{s}(t),\dot{s}(t)\rangle_J=\langle u(t),u(t)\rangle_J;\\
\langle \dot{x}(t),\dot{x}(t)\rangle_J&=&\langle \dot{R}(t)x_0,\dot{R}(t)x_0\rangle_J\\
&=&\langle (u(t)x_0^{\mathbf t}-x_0u^{\mathbf t}(t))Jx_0,(u(t)x_0^{\mathbf t}-x_0u^{\mathbf t}(t))Jx_0\rangle_J\\
&=&\langle u(t),u(t)\rangle_J.
\end{array}
$$

Further we want to show that the curve $t\in I\to R(t)\in \Orth_1(m+1)$ also has the same causal character, with respect to a scalar product in ${\cal G}(n)$ defined below. First, for any matrix $\mathcal A\in {\cal G}(n)$ and $J=\diag(-I_{\nu},I_{n-\nu})$ the Gram matrix,
define the matrix $\mathcal A^J$ by $$ \mathcal A^J:=J\mathcal A^{\mathbf t}J.$$
 ${\cal G}(n)$ may be equipped with a scalar product $\langle\langle .,.\rangle\rangle_J$ of signature $\nu $, defined by
 $\langle\langle \mathcal A,\mathcal B\rangle\rangle_J=\tr{(A^JB)}$. This is positive-definite only for $\nu = 0$. We say that non-zero element $\mathcal A\in {\cal G}(n)$ is timelike if $\langle\langle  \mathcal A, \mathcal A\rangle\rangle_J<0$, it is spacelike if $\langle\langle \mathcal A, \mathcal A\rangle\rangle_J>0$ and it is null if $\langle\langle  \mathcal A, \mathcal A\rangle\rangle_J=0$. The zero element is declared to be spacelike.

Notice that for  $\mathcal A \in {\cal G}(n)$
$$
(\mathcal A^J)^J=\mathcal A,\quad \mbox{and} \quad (\mathcal A\mathcal B)^J=\mathcal B^J\mathcal A^J.
$$
Moreover, if $\mathcal A\in O_{\nu}(n)$, then $\mathcal A^J\mathcal A=\mathcal A\mathcal A^J=\Id$, which implies $\mathcal A^J=\mathcal A^{-1}$.

 We can say that the Lie algebra $\mathfrak{o}_{\nu}(n)$ consists of $(n\times n)$ matrices satisfying $\mathcal A=-\mathcal A^J=-J\mathcal A^{\mathbf t}J$. Consequently, for $\mathcal A\in \mathfrak{o}_{\nu}(n)$, one has
 $$
\mathcal A^J\mathcal A=\mathcal A\mathcal A^J=-\mathcal A^2.
$$
Also, elements in $\mathfrak{o}_{\nu}(n)$ can be written as
$$
\mathcal A=\left(
\begin{array}{cccc}
a_{\nu} & b
& \\
b^{\mathbf t}  &a_{n-\nu}
\end{array}
\right),\qquad a_{\nu}\in\mathfrak{o}(\nu),\quad a_{n-\nu}\in \mathfrak{o}(n-\nu).
$$
So,
\begin{align*}
\langle\langle \mathcal A, \mathcal A\rangle\rangle_J  =\tr(\mathcal A^J\mathcal A) & =-\tr(\mathcal A^2)
\\
&=
\tr\left(
\begin{array}{cccc}
-a_{\nu}^2 & 0
\\
0  &-a_{n-\nu}^2
\end{array}
\right)-2\tr (bb^{\mathbf t}).
\end{align*}
As we see, the first term involving the  skew symmetric matrices $a_{\nu}$ and  $a_{n-\nu}$ is always positive and represents the spacelike part. The matrix $b$ is responsible for the timelike character of elements of the Lie algebra.

We transfer this causal structure to the curves on the group $\Orth_{\nu}(n)$.
Let $A\colon I\to \Orth_{\nu}(n)$ be a smooth curve. We say that the curve $A$ is spacelike, timelike or null if the product $\langle\langle\dot A,\dot A\rangle\rangle_J$ is positive, negative or equals zero, respectively. It can easily be checked that the scalar product $\langle\langle .\,,.\rangle\rangle_J$ is $\Orth_{\nu} (n)$-invariant.

So, using the same ingredients as before and the kinematic equations~\eqref{Kinematic-Lorentz}, we conclude that for the rolling of $S^{m}_1$ on the affine tangent space one get
\begin{equation*}
\langle\langle\dot{R}(t),\dot{R}(t)\rangle\rangle_J=2\langle u(t),u(t)\rangle_J.
\end{equation*}


\subsubsection{\textbf{Controllability}}

We now want to introduce the issue of controllability for this rolling system. The vector function $u(t)$ in the kinematic equations~\eqref{Kinematic-Lorentz} is a control function. The choice of the controls defines the rolling curve. It has been proved in~\cite{KL} that the kinematic equations~\eqref{Kinematic-Lorentz} are completely controllable in  $\dR_{\nu} ^{m+1}\times \Orth_{\nu}^{++}(m+1)$. More general results about sufficient conditions that guarantee the controllability of the rolling process can be found in~\cite{ChKok,Grong}.

However, nothing guarantees that the causal character of the velocity vector remains invariant. The presence of the causal structure arises the natural problem to describe the set $\mathcal R_{x_0}\subset M$ of points reachable by a timelike (spacelike or null) curve from a given point $x_0\in M$. By this we mean that the sign of $\langle\dot{x}(t),\dot{x}(t) \rangle_J$ remains negative (positive or zero), for those $t>0$ where the velocity vector is defined. Before trying to answer this question we analyze a slightly different but simpler issue, that of geodesic reachability by rolling.

\begin{definition}\label{geo_access}
We say that a point $x_1\in M$ is \emph{geodesically reachable by rolling} from another point $x_0\in M$, if there exists a geodesic $x (t)=R(t)x_0$, with $R(t)$ a solution of the second kinematic equation in~\eqref{Kinematic-Lorentz}, satisfying $x (0)=x_0$, $x (t_1)=x_1$, for some $t_1>0$.
\end{definition}
Instead of geodesically reachable by rolling we may simply write geodesically reachable. Since geodesics preserve their causal character it is easier to describe the subset of $\mathcal R_{x_0}$ reachable by geodesics. For the Lorentzian sphere, we characterize the set of points that can be geodesically reachable from a generic point $x_0$. First, we recall from~\cite{KL} what are the geodesics in $ S_1^{m}$ generated by $\mathcal A=(ux_0^{\mathbf t}-x_0u^{\mathbf t})J$ with constant $u$ from the kinematic equations~\eqref{Kinematic-Lorentz}.
\begin{enumerate}
 \item[$\bullet$] If $\langle u,u\rangle_J=1$, then
 $x(t)=\exp(\mathcal At)x_0=x_0\cos(t)+u\sin(t)$
  is a spacelike geodesic satisfying  $x(0)=x_0$.
 \item[$\bullet$]If $\langle u,u\rangle_J=-1$, then
  $x(t)=\exp(\mathcal At)x_0=x_0\cosh(t)+u\sinh(t)$
  is a timelike geodesic satisfying $x(0)=x_0$.
\item[$\bullet$] If $\langle u,u\rangle_J=0$, then
   $x(t)=\exp(\mathcal At)x_0=x_0+ut$
   is a null geodesic satisfying $x(0)=x_0$.
\end{enumerate}
\begin{proposition}\label{geodesics1}
Let $x_0$ be any point in $  S_1^{m}$.  If $x_1\in S_1^{m}$ belongs to the set
$$
\{x\in S_1^{m}, \mbox{ such that } \langle x_0,x\rangle_J>-1\}\cup \{-x_0\},
$$
then $x_1$ is geodesically accessible from $x_0$.
\end{proposition}
\begin{proof}
The proof is constructive, in the sense that we construct the geodesic that realizes the job, according to the value of  $\langle x_0,x_1\rangle_J$.
\begin{enumerate}
  \item If $\langle x_0,x_1\rangle_J>1$, i.e., $\langle x_0,x_1\rangle_J=\cosh \theta$, for some $\theta\neq 0$, the timelike geodesic $x(t)=x_0\cosh(t)+u\sinh(t)$, where
  $u=\frac{x_1-x_0\cosh\theta}{\sinh\theta}$, links $x_0$ (at $t=0$) to $x_1$ (at $t=\theta$). It is a simple calculation to show that, in this case, $\langle u,u\rangle_J=-1$ and we conclude that such kind of points are timelike accessible by geodesics.

   \item If $\langle x_0,x_1\rangle_J=1$,  the null geodesic $x(t)=x_0+tu$, with
  $u=x_1-x_0$, links $x_0$ (at $t=0$) to $x_1$ (at $t=1$). In this case $\langle u,u\rangle_J=0$. Here we have example when $x_1$ is accessible by null geodesics.

   \item If $\langle x_0,x_1\rangle_J\in ]-1,1[$, i.e., $\langle x_0,x_1\rangle_J=\cos \theta$, for some $\theta\neq k\pi$, the spacelike geodesic $x(t)=x_0\cos(t)+u\sin(t)$, where
  $u=\frac{x_1-x_0\cos\theta}{\sin\theta}$, links $x_0$ (at $t=0$) to $x_1$ (at $t=\theta$). In this case, $\langle u,u\rangle_J=1$.

   \item If $x_1=-x_0$, any spacelike geodesic  $x(t)=x_0\cos(t)+u\sin(t)$, with $u$ satisfying $\langle x_0,u\rangle_J=0$, links $x_0$ (at $t=0$) to $x_1=-x_0$ (at $t=\pi$). The last two cases show the accessibility by spacelike geodesics.
\end{enumerate}
\end{proof}
\begin{remark}
We can introduce the time orientation on $S^m_1$ by choosing a globally defined timelike vector field $T$. Then a timelike geodesic starting at $x_0$ and having property $\langle \dot{x}(0),T\rangle_J\, <0$ is called future directed and we introduce the notion of rolling along geodesic to the future. Moreover, when $\langle x_0,x_1\rangle_J\leq -1$ and  $x_1\neq -x_0$, it is possible to reach $x_1$ from $x_0$ by a broken geodesic which change its causal character. For instance, first join $x_0$ to $-x_1$ by a timelike geodesic (if $\langle x_0,x_1\rangle_J>1$) or lightlike geodesic (if $\langle x_0,x_1\rangle_J=1$), and then join   $-x_1$ to $x_1$  by a spacelike geodesic.
\end{remark}
\begin{figure}[h!]\label{figura1}
\centering \scalebox{0.3}{\includegraphics{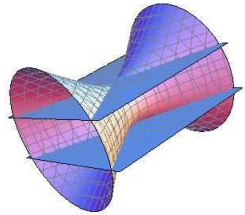}}
\caption{Partition of $S_1^2$ into causal/no-causal subsets from  $x_0=(0,0,1)$.}
\end{figure}

Figure 1 shows the points that can be reached in this way from the point $x_0=(0,0,1)\in S_1^2$. Only $-x_0$ and points above the affine tangent space at $-x_0$ can be reached.

Based in the previous result, it is possible to give  a precise geometric description of the reachable set from a point $x_0$, using spacelike geodesics only, and, similarly, timelike or null geodesics only. Two parallel hyperplanes in $\mathbb{R}_1^{m+1}$, as in the figure above, make the correct separation, as the following shows.
\begin{proposition}\label{geodesics2}
Let $x_0$ and $x_1$ be distinct arbitrary points in $  S_1^{m}$. Then,
\begin{itemize}
\item[(1)]{ $x_1$ is  reachable from $x_0$ by a lightlike geodesic if and only if $x_1\in T_{x_0}^{\aff} S_1^{m}$.}

\item[(2)]{ $x_1$ is  reachable from $x_0$ by a timelike geodesic if and only if $x_1$ is on one side of the hyperplane $T_{x_0}^{\aff} S_1^{m}$, the side that doesn't contain $0\in \dR_1^{m+1}$.}

 \item[ (3)]{ $x_1$ is reachable from $x_0$ by a spacelike geodesic if and only if $x_1$ lies between the hyperplanes $T_{x_0}^{\aff} S_1^{m}$ and
   $T_{-x_0}^{\aff} S_1^{m}$ or $x_1=-x_0$.}
   \end{itemize}
  \end{proposition}
\begin{proof}
The proof is based on some simple facts. First, note that the hyperplanes $T_{x_0}^{\aff} S_1^{m}$ and $T_{-x_0}^{\aff} S_1^{m}$ do not intersect. Otherwise, there would exist $\Omega_1,\Omega_2 \in \orth_1(m+1)$ such that
$$
x_0+\Omega_1x_0=-x_0+\Omega_2x_0\ \  \Leftrightarrow\ \  2x_0+(\Omega_1-\Omega_2)x_0
=0\ \ \Leftrightarrow\ \ x_0=0.
$$
Now observe that the set of points in $ \dR_1^{m+1}$ that satisfy a constraint of the form $\langle x,x_0\rangle_J=k$, for some constant $k$, are hyperplanes. So, using a matching dimension argument and the fact that for any $\Omega \in \orth_1(m+1)$,\linebreak  $\langle x_0+\Omega x_0,x_0\rangle_J=1$, we conclude that (1) is true. And, of course, the set $\{x\in S^m_1: \, \langle x_0,x\rangle_J>1\}$, that can be reached by a timelike geodesic, lies on one side of the hyperplane $\widehat{M}=T_{x_0}^{\aff} S_1^{m}$, the side that doesn't contain the origin, proving (2). The last part is a consequence of the first two and the facts in Proposition \ref{geodesics1}.
\end{proof}
The two images in Figure 2 indicate that the region accessible by spacelike geode\-sics narrows as the point $x_0$ moves away from the origin. For points at infinity, the two hyperplanes coincide and only timelike and null geodesics exist.
\begin{figure}[h!]\label{figura2}
\centering \scalebox{0.3}{\includegraphics{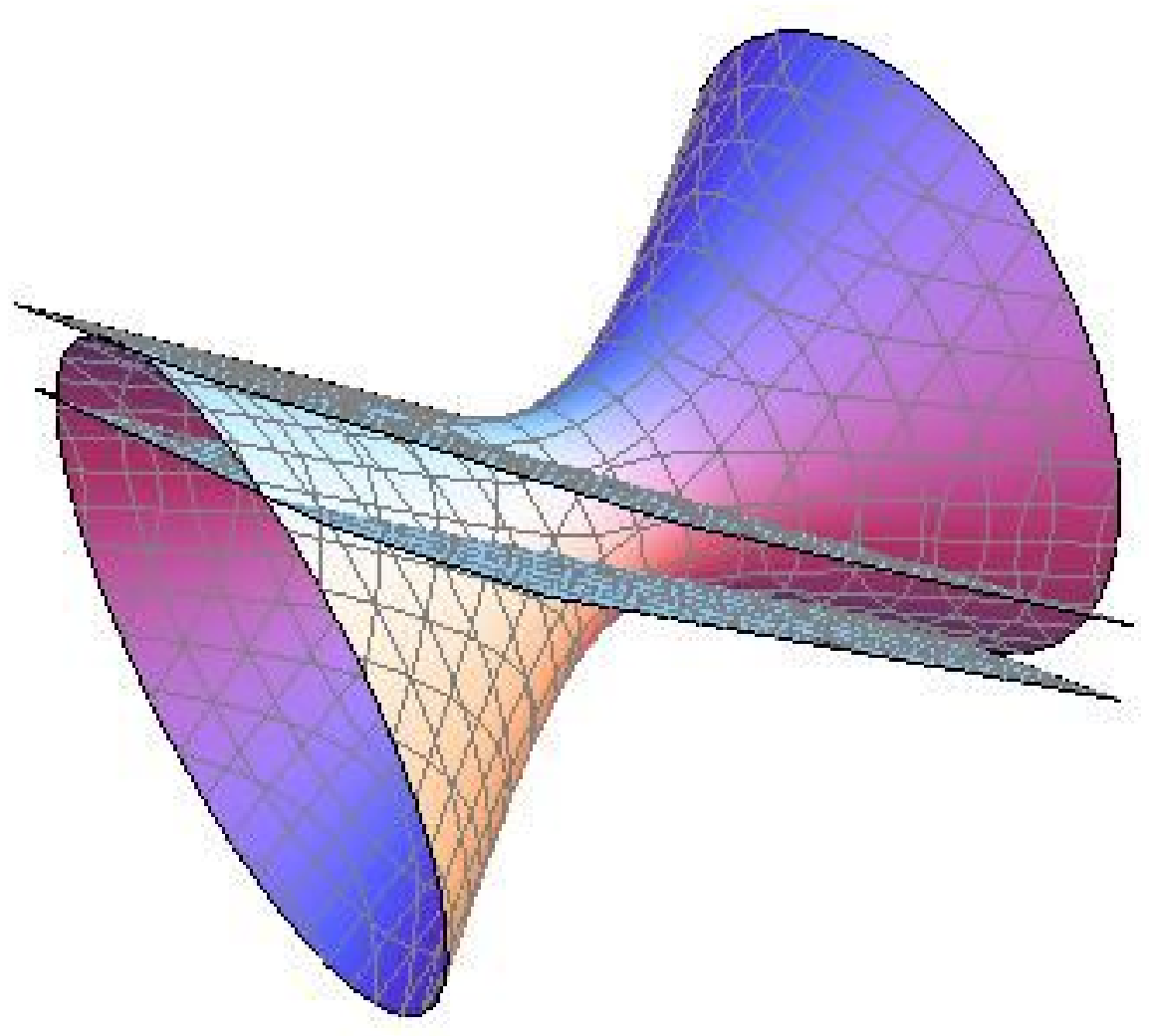}}\hspace*{3 cm}\scalebox{0.26}{\includegraphics{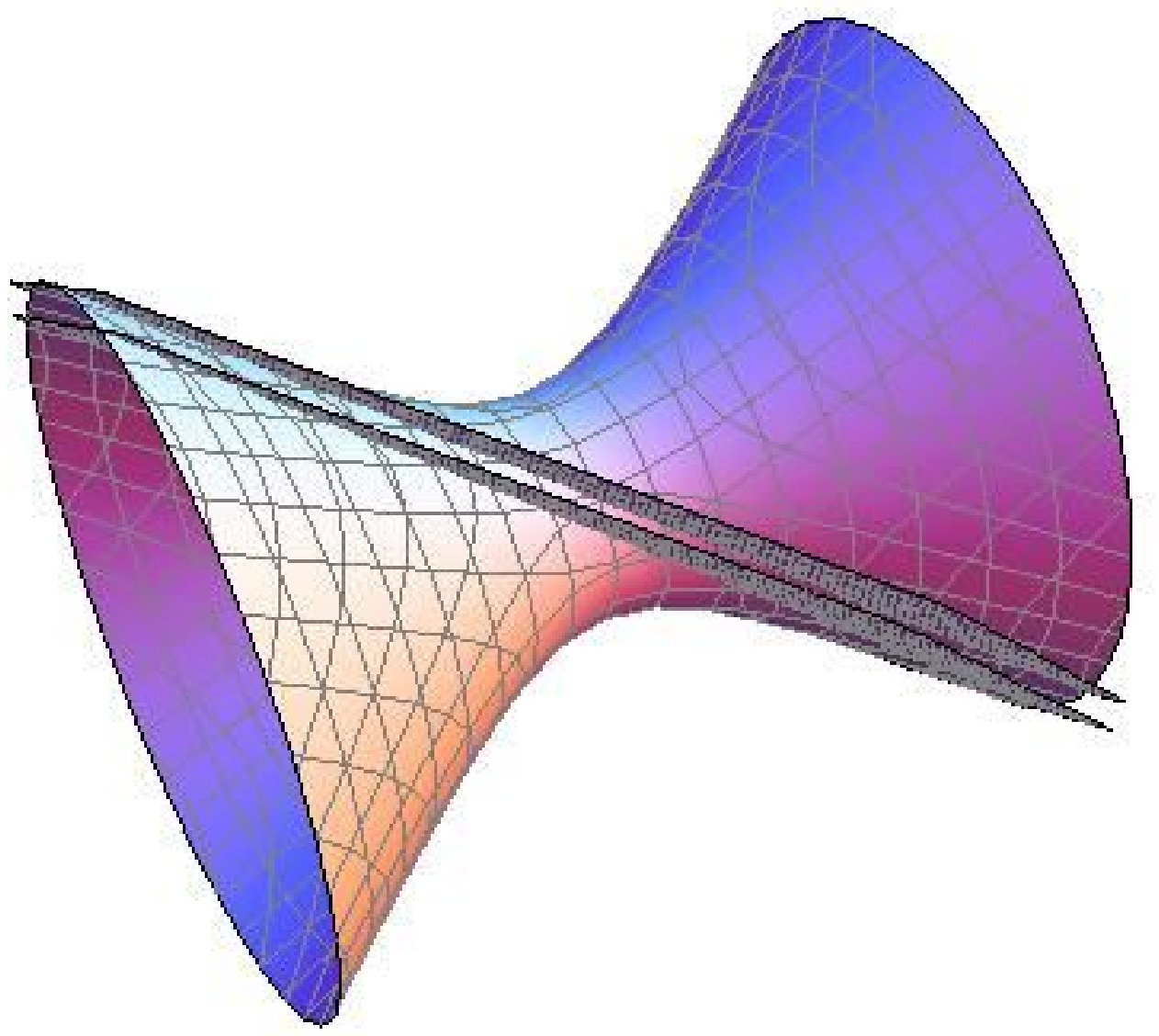}}
\caption{On the left $x_0=(2,2,1)$, on the right $x_0=(\sqrt{40},4,5)$.}
\end{figure}

\subsection{Hyperquadrics} \label{Symmetric spaces}
In this short subsection we want to emphasize that the results obtained for the benchmark example of Section~\ref{Lorentzian sphere} can be extended to hyperquadrics. More precise, let
$$\mathbb S_{\nu}^{m}=\{x\in\mathbb R^{m+1}\mid\ \langle x,x \rangle_J=r,\ r\in\mathbb R\}
$$ be a hypersurface in $\mathbb R^{m+1}$ given by the level set of the scalar product function $\langle \cdot\,,\cdot \rangle_J$ with $J=\diag(-I_{\nu},I_{m+1-\nu})$, that we call a {\it hyperquadric}. The corresponding group of isometries is $\Orth_{\nu}(m+1)$. Analogously to the sphere in Minkowskii space, the kinematic equations of rolling, without slipping and twisting, $\mathbb S_{\nu}^{m}$ over its affine tangent space have the form~\eqref{Kinematic-Lorentz}, as shown in~\cite{ML}. Let $x_0\in\mathbb S_{\nu}^{m}$, $x(t)=R(t)x_0$ and $g(t)=\big(s(t),R^{-1}(t)\big)$, where $R(t)$ is a curve in some subgroup of $\Orth_{\nu}(m+1)$, defined by the orientation and initial value $R(0)$, that jointly with $s(t)$ satisfy the kinematic equations~\eqref{Kinematic-Lorentz}. Then, the map $t\mapsto \big(x(t),g(t)\big)\in \mathbb S_{\nu}^{m}\times\mathbb R_{\nu}^{m+1}\rtimes\Orth_{\nu}(m+1)$ is a rolling of $ \mathbb S_{\nu}^{m}$ over~$T^{\aff}_{x_0} \mathbb S_{\nu}^{m}$.

Now we discuss the issue of the parallel translation. Let $X_0\in T_{x_0}\mathbb S_{\nu}^{m}$ and $X(t)=R(t)X_0$. Then
$$
\langle x(t),X(t) \rangle_J=\langle R(t)x_0,R(t)X_0 \rangle_J=0
$$
since $R(t)$ belongs to the group of isometries and it preserves the corresponding scalar product. It shows that the vector field $X$ is a vector field along the curve $x$ in $\mathbb S_{\nu}^{m}$.
We used an advantage that the manifold is given as a level set of the scalar product and therefore the tangent space is orthogonal to the hypersurface. Then, by using the kinematic equations, we show that $\frac{D}{dt}X(t)=0$ as in the case of the sphere.

Again as in the case of the Lorentzian sphere we can show
$$
\langle \dot{\widehat x}(t),\dot{\widehat x}(t) \rangle_J=\langle \dot x(t),\dot x(t) \rangle_J=\langle\langle \dot R(t),\dot R(t)\rangle\rangle_J=\langle u(t),u(t) \rangle_J
$$
on $\mathbb S_{\nu}^{m}$ and $T_{x_0}^{\aff}\mathbb S_{\nu}^{m}$, that leads to the conclusion that in this case the causal character of the rolling curve $x(t)$, the developing curve $\widehat x(t)$ and the curve in the group of isometries $R(t)$ coincide.

Since the geodesics on $\mathbb S_{\nu}^{m}$ defined by rolling have the same form as in the case of the Lorentzian sphere, Proposition~\ref{geodesics1} remains true for arbitrary hyperquadrics.



\section{Intrinsic rolling} \label{introll}


\subsection{Bundles of isometries} \label{conspace}


Let $V$ and $\widehat V$ be two oriented scalar product spaces with the same index $\mu$ and dimension $m$. We denote by $G(V,\widehat{V})$ the group of all orientation preserving linear isometries between $V$ and $\widehat{V}$. The group $G$ can be any of three groups considered in (\ref{G-orientation}) that preserve orientation, time or space orientation of the scalar product space $V$. When $V=\widehat V$, we write $G(V)$ instead of $G(V,V)$.

For any pair $M$ and $\widehat M$ of connected and oriented manifolds, also with the same index $\mu$ and dimension $m$, we introduce the space $Q$ of all relative positions in which $M$ can be tangent to $\widehat{M}$
\begin{equation} \label{Qdef}
Q = \left\{ \left. q \in G(T_xM, T_{\widehat{x}} \widehat{M})\right| x \in M, \widehat{x} \in \widehat{M}  \right\}.
\end{equation}
This space is a manifold with the structure of an $G_{\mu}(m)$-fiber bundle over $M \times \widehat{M}$ and can be considered as a part of the configuration space of the rolling. The dimension of $Q$ is $2m +(m(m-1)/2=\frac{m(m+3)}{2}$.

Let $\iota: M \rightarrow \mathbb R^{n}_{\nu}=\mathbb{R}^{m + \m}_{\nu}$ and $\widehat{\iota}: \widehat{M} \rightarrow \mathbb{R}^{m + \m}_{\nu}$ be two isometric embeddings. Here $m$ states for the dimension of $T_xM$ and $T_{\widehat x}\widehat M$ and $\mu$ for their index, while $\m$ denotes the dimension of $T_x^{\perp}M$ and $T_{\widehat x}^{\perp}\widehat M$ and $\nu-\mu$ is their index. If $\m> 1$, then the kinematic condition (vi) of normal no twist in Definition~\ref{rolling-classical} becomes non trivial. To describe it we need a counter part of the bundle $Q$, that takes care of the normal components of the embedding. Therefore, we  define a fiber bundle over $M\times \widehat{M}$ of isometries of the normal tangent space. We write
\begin{equation} \label{Pdef}
P_{\iota,\widehat{\iota}} := \left\{ \left. p \in G(T^{\perp}_xM, T^{\perp}_{\widehat{x}} \widehat{M})
\right| x \in M, \widehat{x} \in \widehat{M}  \right\}.
\end{equation}
The space $P_{\iota,\widehat{\iota}}$ is a $G_{\nu-\mu}(\m)$-fiber bundle. We notice that $Q$ is invariant of embeddings $\iota$ and $\widehat\iota$, while $P_{\iota,\widehat{\iota}}$ is not which is reflected in notations. The dimension of $P_{\iota,\widehat\iota}$ is $2m+\frac{\m(\m-1)}{2}$.

We use the notation of the fiber product or Whitney sum $Q \oplus P_{\iota, \widehat \iota}$ for the fiber bundle over $M \times \widehat M$, so that the fiber over $(x, \widehat x) \in M \times \widehat M$ is $Q_{(x,\widehat x)} \times P_{\iota, \widehat \iota (x, \widehat x)}$. The dimension of $Q\oplus P_{\iota,\widehat\iota}$ is $\frac{m(m+3)+\m(\m-1)}{2}$.


\subsection{Reformulation of rolling in terms of bundles}


We define the rolling by making use of the bundle $Q \oplus P_{\iota, \widehat \iota}$ and then we show that the new definition is equivalent to Definition~\ref{rolling-classical}.

\begin{definition} \label{imbeddef3}
An extrinsic rolling $($without slipping or twisting$)$ of $M$ on $\widehat{M}$  along $x(t)$ and $\widehat{x}(t)$  is an absolutely continuous curve $(q,p):[0,\tau] \to Q \oplus
P_{\iota,\widehat{\iota}}$ such that $(q(t),p(t))$ satisfies
\begin{itemize}
\item[(I)] no slip condition: $\dot{\widehat{x}}(t) = q(t) \dot{x}(t)$ for almost every $t$,
\item[(II)] no twist condition $($tangential part$):$
$q(t) \frac{D}{dt} Z(t) = \frac{D}{dt} q(t) Z(t)$
for any tangent vector field $Z(t)$ along $x(t)$ and almost every $t$,
\item[(III)] no twist condition $($normal part$):$ $p(t) \frac{D^{\perp}}{dt} \Psi(t) =
\frac{D^{\perp}}{dt} p(t) \Psi(t)$ for any normal vector field $\Psi(t)$ along $x(t)$ and almost every $t$.
\end{itemize}
\end{definition}

The following proposition shows the equivalence of Definitions~\ref{rolling-classical} and~\ref{imbeddef3}.

\begin{proposition} \label{Prop gtoqp}
If a curve $(x,g)\colon [0,\tau] \rightarrow M \times \mathbb R^{m+\m}_{\nu}\rtimes G_{\nu}(m+\m)$ satisfies the conditions {\rm (i)-(vi)} in Definition~\ref{rolling-classical}, then the mapping
$$t \mapsto (d_{x(t)}g(t)|_{T_{x(t)}M},d_{x(t)}g(t)|_{T_{x(t)}M^{\perp}}) = : \left(q(t),p(t) \right), $$
defines a curve in $Q \oplus P_{\iota,\widehat{\iota}}$ satisfying the conditions {\rm (I)-(III)} of Definition~\ref{imbeddef3}.

Conversely, if $(q,p)\colon [0,\tau] \to Q \oplus P_{\iota,\widehat{\iota}}$ is an absolutely continuous curve  satisfying {\rm (I)-(III)}, then there exists a unique rolling
$$(x,g):[0,\tau] \to M \times \mathbb R^{m+\m}_{\nu}\rtimes G_{\nu}(m+\m),$$
such that $d_{x(t)}g(t)|_{T_{x(t)} M} =q(t)$ and $d_{x(t)}g(t)|_{T_{x(t)} M^{\perp}} =p(t)$.
\end{proposition}

\begin{proof} We sketch the proof, since details can be found in~\cite{GGLM}.

The conditions (i)-(iii) ensures that the map $(q,p)$ preserves the group $G$ of the orientation and it is clear that the conditions (I)-(III) correspond to the conditions (iv)-(vi).

Conversely, for a given curve $(q(t),p(t))$ in $Q \oplus P_{\iota,\widehat{\iota}}$ over $(x(t),\widehat{x}(t))\in M \times \widehat{M}$ the isometry $g \in \mathbb R^{m+\m}_{\nu}\rtimes G_{\nu}(m+\m)$ is defined by
$$
g(t)\colon \bar{x} \mapsto \bar{A}(t) \bar{x} + \bar{r}(t),  \quad \bar{A}(t)\in G_{\nu}(m+\m)
$$
where $\bar{A}(t)=d_{x(t)}g(t)$ is determined by the conditions
$$
d_{x(t)}g(t)|_{T_{x(t)} M} =q(t)|_{T_{x(t)} M}, \qquad  d_{x(t)}g(t)|_{T_{x(t)} M^{\perp}} = p(t)|_{T_{x(t)} M^{\perp}}.
$$
and $\bar{r}(t) = \widehat{x}(t) - \bar{A}(t) x(t)$.
\end{proof}

A purely intrinsic definition of a rolling is deduced from Definition~\ref{imbeddef3}, by restricting it to the bundle~$Q$.

\begin{definition} \label{intrinsdef}
An intrinsic rolling $($without slipping or twisting$)$ of  $M$ over a manifold $\widehat{M}$, along curves $x(t)$ and $\widehat{x}(t)$,  is an absolutely continuous curve $q\colon[0,\tau] \rightarrow Q$,  with projections $x(t) = \pr_M q(t)$ and \ $\widehat{x}(t) = \pr_{\widehat{M}} q(t)$, satisfying the following conditions:
\item[(I')] no slip condition: $\dot{\widehat{x}}(t) = q(t) \dot{x}(t)$ for almost all $t$,
\item[(II')] no twist condition: $Z(t)$ is a parallel tangent vector field along $x(t)$, if and only if $q(t) Z(t)$ is parallel along $\widehat{x}(t)$ for almost all $t$.
\end{definition}


\subsection{Expression of $(q,p)$ in parallel frame}


Since the rolling without twisting preserves parallel vector fields, we expect that the expression of the curve $(q,p)\colon [0,\tau] \to Q \oplus P_{\iota,\widehat{\iota}}$ would be simpler in parallel frames. Let $x\colon[0,\tau]\to M$ and $\widehat x:[0,\tau]\to\widehat M$ be two fixed curves. We denote by $\{e_j(t)\}_{j =1}^m$ an orthonormal  frame field of parallel tangent vector fields along~$x(t)$ and by $\{\epsilon_\lambda(t)\}_{\lambda =1}^{\m}$ an orthonormal  frame field of normal parallel vector fields along~$x(t)$. Such vector fields can be constructed by parallel transport and normal parallel transport along~$x(t)$. Similarly, along $\widehat x(t)$, we define parallel frames $\{\hat{e}_i\}_{i =1}^m$ and~$\{\hat{\epsilon}_\kappa\}_{\kappa =1}^{\m}$.

\begin{lemma} \label{constantM}
A curve $(q(t),p(t))$ in $Q \oplus P_{\iota,\widehat{\iota}}$ in the fibers over $(x(t),\widehat{x}(t))$, satisfies (II) and (III) if and only if the matrices
$$A(t)=\{a_{ij}(t)\} = \{\langle{\hat{e}}_i, q(t) e_j \rangle_J\}, \quad B(t)=\{b_{\kappa \lambda}(t)\} = \{\langle\hat{\epsilon}_\kappa(t), p(t) \epsilon_\lambda(t)\rangle_J\},$$ in parallel frames are constant.
\end{lemma}

\begin{proof} Since $\{q(t)e_j \}$, $j=1,\ldots, m$ is a parallel frame along $\widehat x(t)$, then the coordinates of vectors $\{q(t)e_j \}$ in the basis $\{\hat{e}_i\}$, $i=1,\ldots, m$, should be constant. The precise calculation go along the same lines as those in~\cite{GGLM} for the Riemannian case.
\end{proof}

\begin{example}\label{meaning-for-Lorentz} We illustrate Lemma~\ref{constantM} by constructing the matrices $A, B$ for the case of the $2$-dimensional Lorentz sphere and give them a geometric meaning. The notations are those in Section~\ref{Lorentzian sphere}.
Let $x_0=\left[\begin{array}{ccc}
0&0&1
\end{array}\right] ^{\mathbf t} \in S_1^2$,
$$
\begin{array}{l}
R(t)=\exp{(t\left[\begin{array}{ccc}
0&0&1\\
0&0&0\\
1&0&0
\end{array}\right])}=\left[\begin{array}{ccc}
\cosh{t}&0&\sinh{t}\\
0&1&0\\
\sinh{t}&0&\cosh{t}
\end{array}\right] , \\\\ x(t)=R(t)x_0=\left[\begin{array}{c}
\sinh{t}\\
0\\
\cosh{t}
\end{array}\right] .
\end{array}
$$
Define
$$
\begin{array}{l}
e_1(t)=\left[\begin{array}{c}
\cosh{t}\\0\\\sinh{t}
\end{array}\right] , \,\, e_2(t)=\left[\begin{array}{c}
0\\1\\0
\end{array}\right] , \,\, \epsilon_1(t)=\left[\begin{array}{c}
\sinh{t}\\0\\\cosh{t}
\end{array}\right].
\end{array} $$
The frame field $\{e_1(t), e_2(t)\}$ is orthonormal parallel and tangent along $x(t)$ and the vector field $\{\epsilon_1(t)\}$ represents the normal parallel vector field along $x(t)$. Note that $e_1(t)$ is timelike, while $e_2(t)$ and $\epsilon_1(t)$ are spacelike. Now, $\widehat{M}=T_{x_0}^{\aff}M$, so that  $T_{\widehat{x}(t)} \widehat{M}=T_{x_0}S_1^2$ and, consequently, $\{\hat{e}_1(t), \hat{e}_2(t)\}$ (respectively  $\{\hat{\epsilon}_1(t)\}$), defined below, form an orthonormal frame field of parallel  tangent (respectively normal) vector fields along $\hat{x}(t)$.
$$
\hat{e}_1(t)=\left[\begin{array}{c}\sqrt{2}\\1\\0\end{array}\right] ,\quad \hat{e}_2(t)=\left[\begin{array}{c}1 \\\sqrt{2}\\0\end{array}\right]
,\quad \hat{\epsilon}_1(t)=\left[\begin{array}{c}0 \\0\\1\end{array}\right] .
$$
Again, the first vector is timelike, while the last two are spacelike. To compute the matrices $A$ and $B$ in the previous Lemma, note that
$$
q(t)e_1(t)=R^{-1}(t)e_1(t)=\left[\begin{array}{c}
1\\0\\0
\end{array}\right],\quad q(t)e_2(t)=R^{-1}(t)e_2(t)=\left[\begin{array}{c}
0\\1\\0
\end{array}\right],
$$
and because we are in codimension $1$, $p(t)\epsilon_1(t)=\hat{\epsilon}_1(t)$. So, $B=I_1$ and the matrix $A$ with entries $a_{ij}=\langle\hat{e}_i(t), q(t) e_j(t)\rangle_J$ is
$$
A=\left[\begin{array}{cc}
-\sqrt{2}&1\\
-1&\sqrt{2}
\end{array}\right].
$$
We emphasize that since we use the scalar product $\langle\cdot\,, \cdot\rangle_J$ defined by $J$, then the matrix $A$ defers from
matrix $T$ defined by usual euclidean inner product by the first row: all entries of the first row of $T$ have opposite sign to the corresponding entries of the first row of $A$, that is, $A=JT$. This is due to the fact that $\hat{e}_1$ is timelike.  If the basis elements where all spacelike, then $A$ and $T$ would coincide. We conclude that the block matrix
$$
W=\left[\begin{array}{cc}
A&0\\
0&B\\
\end{array}\right]=\left[\begin{array}{ccc}
-\sqrt{2}&1&0\\
-1&\sqrt{2}&0\\
0&0&1
\end{array}\right]
$$
is a twist which reverses time-orientation and preserves space-orientation. In particular, $W$ transforms  the ordered orthonormal basis $\{\hat{e}_1,\hat{e}_2, \hat{\epsilon}_1\}$ of $\mathbb{R}_1^3$ into the ordered orthonormal basis $\{-q(t)e_1(t),q(t)e_2(t), \hat{\epsilon}_1\}$.
\end{example}



\subsection{Intrinsic and extrinsic rollings along the same curves}\label{exvsint}


Assume that a pair of curves $(x,\widehat x)\colon[0,\tau]\to M\times\widehat M$ is fixed and they are projections of an intrinsic rolling map $q(t)$. The following uniqueness question can be asked: are there other intrinsic rollings along the same curve $(x,\widehat x)$?

Before giving the answer to this question, we make some observations. Let $\{e_j(t)\}_{j =1}^m$ and $\{\widehat e_j(t)\}_{j =1}^m$ be orthonormal tangent parallel frames along $x(t)$ and $\widehat x(t)$, respectively. Then
$$
\dot{\widehat x_i}(t)=\langle \widehat e_i,\dot{\widehat x}(t)\rangle_J=\langle \widehat e_i,q\dot x (t)\rangle_J=\sum_{j=1}^n\dot x_j(t)\langle \widehat e_i,qe_j\rangle_J=\sum_{j=1}^n a_{ij}\dot x_j(t).
$$
If we assume that some of $\dot x_k=\langle e_k,\dot x\rangle_J$ vanish, then $k$-columns of the matrix $\{a_{ij}\}$ can be changed without influence on the resulting $\dot{\widehat x}(t)$ and this gives the freedom in the choice of the intrinsic rolling $q(t)$. Now we introduce some necessary definitions and formulate the result.

Recall that a tangent vector field $v$ along an absolutely continuous curve $\gamma$ on a pseudo-Riemannian manifold is called normal to $\gamma$ if $\langle v(t),\dot\gamma(t)\rangle_J=0$ for almost all $t$ from the domain of the definition of the curve. We understand that it could be confusing to use the word {\it normal} in two different meanings, nevertheless since both meanings are classical and we use the latter sense of normal vector field only in Theorem~\ref{number intrinsic}, we continue to do it.

Let $q\colon[0,\tau] \rightarrow Q$ be an intrinsic rolling map with projection $\pr_{M \times \widehat{M}} q(t) = (x(t),\widehat{x}(t))$.
Define the vector spaces
$$V = \Big\{v(t) \text{ is a tangent parallel vector field normal to } x(t) \Big\},$$
$$\widehat{V} = \Big\{\widehat{v}(t) \text{ is a tangent parallel vector field normal to } \widehat{x}(t)\Big\}.$$
 Note that both, the inner product and the orientation, are preserved under parallel transport. Hence, for any pair $v, w \in V$, the value of $\langle v(t), w(t) \rangle_J$ remains constant for any $t$. Therefore, the metric on $M$ induces a well defined inner product on $V$. Similarly, the $G$-orientation on $V$ is well defined, since it does not depend on $t$. Analogous considerations hold for $\widehat{V}$.

\begin{theorem} \label{number intrinsic}
Let $q\colon[0,\tau] \rightarrow Q$ be a given intrinsic rolling map without slipping or twisting that is projected to $(x(t),\widehat{x}(t))$.
Then $\dim V=\dim\widehat V$. Moreover,
\begin{itemize}
\item[(a)] the map $q$ is the unique intrinsic rolling of $M$ over $\widehat{M}$ along $x(t)$ and $\widehat{x}(t)$ if and only if $ \dim V\leq 1$,
\item[(b)] if $\dim V \geq 2$, all the rolling maps along $x(t)$ and $\widehat{x}(t)$ differ from $q$ by an element in $G(\widehat{V})$.
\end{itemize}
\end{theorem}

\begin{proof} Choose the frame of parallel vector fields $\{e_i\}_{i=1}^n$ along $x(t)$ and define the parallel frame $\{\hat{e}_i \}_{i=1}^n$ along $\widehat{x}(t)$ by $q(t) e_i = \hat{e}_i$. Assume that the first $k$ vector fields of each frame are orthogonal to curves $x(t)$ and $\widehat x(t)$, respectively. Notice that $e_1, \dots, e_k$ is a basis for $V$, and $\hat e_1, \dots, \hat e_k$ is a basis for $\widehat V$. By Lemma~\ref{constantM} the corresponding matrix $A=\{a_{ij}\}=\langle\hat{e}_i,q e_j\rangle_J$ is the diagonal matrix with $\pm 1$ on diagonal according to the causal character of the basic vectors.

Writing $\dot{\widehat{x}} = \sum_{i =1}^n \dot{\widehat{x}}_i(t) \hat{e}_i(t)$ and $\dot{x} = \sum_{i =1}^n \dot{x}_i(t) e_i(t)$, we get $\dot{\widehat{x}}_i(t) = \dot{x}_i(t)$ for $i=1,\ldots,n$ and $\dot{\widehat{x}}_i(t) = \dot x_i(t) =0$ for $i=1,\ldots,k$. So, if $\widetilde q$ is any other rolling, then $\widetilde A =\{\widetilde a_{ij}\}= \langle\hat{e}_i(t) ,\widetilde q(t) \, e_j(t)  \rangle_J$ is clearly of the form
\begin{equation}\label{suchmatrix} \widetilde A =
\left( \begin{array}{cc} A' & 0 \\ 0 & J_{m-k} \end{array} \right), \qquad A' \in G(\mathbb R_{\xi}^{k}),\end{equation}
where $J_{m-k}$ is the $\big((m-k)\times(m-k)\big)$ matrix with entires $\pm1$ on the diagonal and $0$ otherwise. The matrix $A^{\prime}$ is unique if $k$ is 0 or 1. If $k \geq 2$, there is more freedom, since in the equality $\dot{\widehat x_i}=\sum_{j=1}^n\widetilde a_{ij}\dot x_j=\sum_{j=1}^n\widetilde a_{ij}\dot{\widehat x}_j$ the first $k$ values of $\dot{\widehat x}_j$ vanish.

The converse also holds, that is, for any matrix $\widetilde A$ on the form~\eqref{suchmatrix}, there is a rolling $\widetilde q$ corresponding to it.
\end{proof}

In particular, if the curve $x\colon[0,\tau]\to M$ is a geodesic, we have the following consequence of Theorem~\ref{number intrinsic}.

\begin{corollary}\label{answer1}
Assume that $x(t)$ is a geodesic in $M$. Then there exists an intrinsic rolling of $M$ on $\widehat{M}$ along $(x(t),\widehat{x}(t))$ if and only if $\widehat{x}(t)$ is a geodesic such that $\langle \dot x(t),\dot x(t)\rangle_J=\langle \dot{\widehat x}(t),\dot{\widehat x}(t)\rangle_J$.
Moreover, if $m \geq 2$, and if $\widehat{V}$ is defined as in Theorem \ref{number intrinsic}, then
$$\dim \widehat{V} = m-1,$$
and all the rollings along $x(t)$ and $\widehat{x}(t)$ differ by an element in $G(\widehat{V})$.
\end{corollary}

\begin{proof} Calculating the covariant derivatives, and using the no-slip and no-twist conditions (I')-(II'), we obtain
$$\frac{D}{dt} \dot{\widehat{x}}(t) = \frac{D}{dt} q(t) \dot{x}(t)=q(t) \frac{D}{dt} \dot{x}(t).
$$
Thus, the curve $x(t)$ is a geodesic if and only if $\widehat{x}(t)$ is also geodesic. The property (I') implies $\langle \dot{\widehat x}(t),\dot{\widehat x}(t)\rangle_J=\langle \dot x(t),\dot x(t)\rangle_J$. Conversely, equal speeds implies that $\dot x(t)$ differs from $\dot{\widehat x}(t)$ by an isometry $q(t)\colon T_{x(t)}M\to T_{\widehat x}(t)\widehat M$ and the condition~(I') follows.

Without loss of generality we can suppose that $x$ is a timelike geodesic. We start the construction of rolling map by choosing the vector field $e_m(t) = \frac{\dot x(t)}{|\langle \dot{x}(t), \dot{x}(t) \rangle_J|^{1/2}}$ that is parallel along $x(t)$. Pick the remaining $m-1$ parallel vector fields such that they form an orthonormal basis together with $e_m(t)$ along the curve $x(t)$. We repeat the same construction for a parallel frame $\{ \hat{e}_i(t) \}_{i=1}^m$ along $\widehat{x}(t)$. Define the intrinsic rolling $q(t)$ by
\begin{equation} \label{intr1 eq}
\begin{split}
& \langle\hat{e}_m(t), q(t) \, e_j(t)\rangle_J = \langle\hat{e}_j(t) , q(t) \, e_m(t)\rangle_J = -\delta_{m,j},\\
& \ \ \ \ \ \ A' = \left\{\langle\hat{e}_{i}(t) , q(t) \, e_{j}(t)\rangle_J \right\}_{i,j =1}^{m-1},
\end{split}
\end{equation}
where $A' \in G_{\mu-1}(m-1)$ will be a constant matrix. Conversely, we can construct a rolling
by formulas~\eqref{intr1 eq} starting from $A' \in G_{\mu-1}(m-1)$.
\end{proof}
\begin{remark}
Analogously to the spaces $V$ and $\widehat V$ in Theorem~\ref{number intrinsic}, let us define the vector spaces
$$E = \left\{\epsilon(t) \text{ is a normal parallel vector field normal to } x(t) \right\},$$
$$\widehat{E} = \left\{\widehat{\epsilon}(t) \text{ is a normal parallel vector field normal to } \widehat{x}(t) \right\},$$
with inner product and orientation induced from the metrics on $T_x^{\perp}M$ and $T_{\widehat x}^{\perp}\widehat M$. Both vector spaces have dimension $\m$. An extrinsic rolling $(q,p)$ extending an intrinsic rolling $q$ is determined up to a left action of $G(\widehat{E})$ or, equivalently, up to a right action of $G(E)$. Both $G(E)$ and $G(\widehat{E})$ are isomorphic to $G_{\nu-\mu }(\m)$, but not canonically.
\end{remark}

The following theorem concerns the question of the extension of an intrinsic rolling to the extrinsic one if the isometric imbeddings of $M$ and $\widehat M$ into some ${\mathbb R}_{\nu}^{n}$ are given.

\begin{theorem} \label{unique extension}
Let $q\colon[0,\tau] \rightarrow Q$ be an intrinsic rolling and let $\iota\colon M \rightarrow \mathbb{R}^{m+\m}_{\nu}$ and $\widehat{\iota}\colon \widehat{M} \rightarrow \mathbb{R}^{m+\m}_{\nu}$ be given embeddings.
Then, given an initial normal configuration
$$p_0 \in (P_{\iota,\widehat \iota})_{(x_0, \widehat x_0)},\mbox{ where }(x_0, \widehat x_0) = \pr_{M \times \widehat M} q(0),$$
there exists a unique extrinsic rolling $(q,p): [0,\tau] \rightarrow Q \oplus P_{\iota,\widehat{\iota}}$ satisfying $p(0) = p_0$.
\end{theorem}

\begin{proof}
Let $B_0 \in G_{\nu-\mu}(\m)$ be defined by
$B_0 = (b_{\kappa\lambda}) =\left(\langle\widehat{\epsilon}_\kappa(0) , p_0 \, \epsilon_\lambda(0) \rangle_J \right)$,
with $\{\epsilon_\lambda(t)\}_{\lambda =1}^\nu$ and $\{\hat{\epsilon}_\kappa(t)\}_{\kappa =1}^\nu$ normal parallel frames along $x(t)$ and $\widehat{x}(t)$, respectively.
Then $p(t)$ satisfies
$b_{\kappa\lambda} = \langle\widehat{\epsilon}_\kappa(t) , p(t) \, \epsilon_\lambda(t) \rangle_J$,
by Lemma~\ref{constantM}, and it is uniquely determined by this.
\end{proof}


\section{Distributions  associated to intrinsic rolling}\label{sec:dist}

The aim of this Section is to formulate the kinematic conditions for rolling without slipping and without twisting in terms of a distribution or subbundle of $T(Q \oplus P_{\iota,\widehat{\iota}})$. In this setting, a rolling will be an absolutely continuous curve in the configuration space $Q \oplus P_{\iota,\widehat{\iota}}$ tangent to the distribution almost everywhere. Namely, the kinematic conditions of no-slip and no-twist will force this curve to be tangent to the distribution.


\subsection{Local trivialization of $Q\oplus P_{\iota,\widehat{\iota}}$ and the tangent space of $G_{\mu}(m)$} \label{local triv}

Let
\begin{equation}\label{bundle_QP}
\pi\colon Q \oplus P_{\iota,\widehat{\iota}} \to M \times \widehat{M}
\end{equation}
denote the bundle for the rolling map. Consider a rolling curve
$$
\gamma = (q, p)\colon I\to Q \oplus P_{\iota,\widehat{\iota}}
$$
and assume that the interval $I$ is so small that $\pi\circ\gamma(I)=\left( x(I) , \widehat{x}(I) \right)\in U\times \widehat U$ and $U\in M$, $\widehat U\in\widehat M$ are chosen such that the bundle~\eqref{bundle_QP} trivializes when restricted to the domain $U \times \widehat{U}$. Thus there is a diffeomorphism $h$ defining the trivialization
\begin{equation} \label{trivi}
\begin{array}{rcl}
Q \oplus P_{\iota,\widehat{\iota}}\supset \pi^{-1}(U \times \widehat{U}) & \stackrel{h}{\to} & U \times \widehat{U} \times G_{\mu}(m) \times G_{\nu-\mu}(\m)
\\
(q(t), p(t)) & \mapsto & (x(t),\widehat{x}(t), A(t), B(t)),
\end{array}
\end{equation}
given by projections
$$x(t) = \pr_U (q(t),p(t)), \qquad \widehat{x}(t) = \pr_{\widehat{U}} (q(t),p(t)),\qquad t\in I,$$
$$A = \left( a_{ij} \right)_{i,j = 1}^m = \left( \langle q e_j, \hat{e}_i \rangle_J \right)_{i,j = 1}^m, \qquad
B = \left( b_{\kappa\lambda} \right)_{\kappa,\lambda = 1}^{\m} =
\left( \langle p \epsilon_\lambda, \hat{\epsilon}_\kappa \rangle_J \right)_{\kappa,\lambda = 1}^{\m}.$$
Here $\{ e_j \}_{j = 1}^m$, $\{ \epsilon_\lambda \}_{\lambda = 1}^\m$, $\{ \hat{e}_i \}_{i = 1}^m$ and $\{ \hat{\epsilon}_\kappa \}_{\kappa = 1}^\m$ are oriented orthonormal frames of vector fields of $TM|_U$, $T^\perp M|_U$, $T\widehat{M}|_{\widehat{U}}$ and $T^\perp\widehat{M}|_{\widehat{U}}$, respectively. Moreover, we assume that the first $\mu$ terms of $\{ e_j \}_{j = 1}^m$ and $\{ \hat{e}_i \}_{i = 1}^m$ are timelike. Corres\-pondingly, the first $\nu-\mu$ vector fields $\{ \epsilon_\lambda \}_{\lambda = 1}^\m$ and $\{ \hat{\epsilon}_\kappa \}_{\kappa = 1}^\m$ are also timelike.
The groups $G_{\mu}(m)$ and $G_{\nu-\mu}(\m)$ are chosen according to the desirable $G$-orientation properties of the rolling.

The kinematic conditions (I)-(III) are written as restrictions on the velocity vector
\begin{align*}
\dot{\gamma}(t)&=\big(\dot x(t), \dot{\widehat{x}}(t), \dot A(t), \dot B(t))\in T_{\gamma(t)}\pi^{-1}(U \times \widehat{U})
\\
&\cong T_{x(t)}U \times T_{\widehat x(t)}\widehat{U} \times T_{A(t)}G_{\mu}(m) \times T_{B(t)}G_{\nu-\mu}(\m).
\end{align*}

We recall the description of the tangent space  $TG_{\mu}(m)$ in terms of left and right invariant vector fields. The tangent space at the identity of $G_{\mu}(m)$, or the Lie algebra $\mathfrak g_{\mu}(m)$, is spanned by
\begin{equation}\label{basis}
\{ W_{ij}= \frac{\partial}{\partial a_{ij}} - \varepsilon_{i}\varepsilon_{j}\frac{\partial}{\partial a_{ji}}, \quad 1 \leq i < j \leq m\},
\end{equation}
 where
$$
\varepsilon_i=
\begin{cases}
-1 &\quad\text{if}\quad 1\leq i\leq\mu,
\\
1 &\quad\text{if}\quad \mu+1\leq i\leq m,
\end{cases},\quad\langle e_i,e_j\rangle_J=\langle \hat e_i,\hat e_j\rangle_J=\varepsilon_i\delta_{ij},
$$
and $\delta_{ij}$ is the Kronecker symbol.
The left and right invariant vector fields obtained by the translations of $W_{ij}(1)$ in (\ref{basis}) by $A\in G_{\mu}(m)$ are the following
\begin{equation} \label{SObasis}
A \cdot W_{ij}(1) = \sum_{r =1}^m \left(a_{ri} \frac{\partial}{\partial a_{rj}} - \varepsilon_{i}\varepsilon_{j}a_{rj} \frac{\partial}{\partial a_{ri}} \right)
\end{equation}\begin{equation} \label{SObasis_right}
W_{ij}(1) \cdot A= \sum_{r =1}^m \left(a_{js} \frac{\partial}{\partial a_{is}} - \varepsilon_{i}\varepsilon_{j}a_{is} \frac{\partial}{\partial a_{js}} \right).
\end{equation} See Appendix for the details.


\subsection{Distributions}


Now we are ready to rewrite the kinematic conditions (I)-(III) as a distribution over $ Q \oplus P_{\iota,\widehat{\iota}} $. Consider the image of $\gamma(t)$, satisfying the conditions (I)-(III), under the trivialization. Then
\begin{equation}\label{dotgamma}
\dot{\gamma}(t) = \dot{x}(t) + \dot{\widehat{x}}(t) + \sum_{i,j = 1}^m \dot{a}_{ij} \frac{\partial}{\partial a_{ij}}
+ \sum_{\kappa,\lambda = 1}^{\m} \dot{b}_{\kappa\lambda} \frac{\partial}{\partial b_{\kappa\lambda}}.
\end{equation}

We want to write the last two terms in~\eqref{dotgamma} in the left invariant bases of $TG_{\mu}(m)$ and $TG_{\nu-\mu}(\m)$, based on conditions (II) and (III).
We start from (II) and recall that according to coordinate representation of $A=\{a_{ij}\}=\{\langle q(t)e_j,\hat e_i\rangle_J\}$ in orthonormal bases $\{e_j\}_{j=1}^{n}$ and $\{\widehat e_j\}_{j=1}^{n}$ we obtain
$$q(t) e_j = \sum_{l =1}^m \varepsilon_l a_{lj}(t) \hat{e}_l,\quad\text{and}\quad
q^{-1}(t) \hat{e}_i =  \sum_{l=1}^m\varepsilon_l a_{il}(t)e_l.
$$
Condition (II) holds if and only if  $q \frac{D}{dt} e_j(x(t)) = \frac{D}{dt} qe_j(x(t))$,
which yields
\small{\begin{multline*}
0 = \left\langle q \frac{D}{dt} e_j(x(t))- \frac{D}{dt} q e_j(x(t)), \hat{e}_i \right\rangle_J
\\
= \left\langle \nabla_{\dot x(t)} e_j, q^{-1} \hat{e}_i \right\rangle_J
- \left\langle \sum_{l=1}^m \varepsilon_l\dot{a}_{lj} \hat{e}_l , \hat{e}_i \right\rangle_J
- \left\langle \sum_{l=1}^m \varepsilon_la_{lj} \nabla_{\dot{\hat x}(t)} \hat{e}_l, \hat{e}_i \right\rangle_J
\\
= \sum_{l = 1}^m \varepsilon_la_{il} \left\langle \nabla_{\dot x(t)} e_j, e_l \right\rangle_J
- \dot{a}_{ij} - \sum_{l=1}^m \varepsilon_la_{lj} \left\langle\nabla_{\dot{\hat x}(t)} \hat{e}_l, \hat{e}_i \right\rangle_J,
\end{multline*}}
for every $i,j =1, \dots, m$. Hence, the third term in~\eqref{dotgamma} can be written as
\begin{multline}\label{sumdist}
\hspace*{-0.4 cm}\sum_{i,j=1}^m \dot{a}_{ij} \frac{\partial}{\partial a_{ij}}
= \sum_{i,j=1}^m \left( \sum_{l = 1}^m\varepsilon_l a_{il} \left\langle \nabla_{\dot x(t)} e_j, e_l \right\rangle_J
 - \sum_{l=1}^m \varepsilon_la_{lj} \left\langle\nabla_{q\dot x(t)} \hat{e}_l, \hat{e}_i \right\rangle_J\right) \frac{\partial}{\partial a_{ij}}
 \\
= \sum_{j,l=1}^m \left\langle \nabla_{\dot x(t)} e_j, e_l \right\rangle_J \varepsilon_l\sum_{i=1}^ma_{il}\frac{\partial}{\partial a_{ij}}
 - \sum_{i,l=1}^m \left\langle\nabla_{q\dot x(t)} \hat{e}_l, \hat{e}_i \right\rangle_J \varepsilon_l\sum_{j=1}^{m}a_{lj}\frac{\partial}{\partial a_{ij}}
 \\
 = \sum_{j,l=1}^m \left\langle \nabla_{\dot x(t)} e_j, e_l \right\rangle_J \varepsilon_lA \cdot \frac{\partial}{\partial a_{lj}}
 - \sum_{i,l=1}^m \left\langle\nabla_{q\dot x(t)} \hat{e}_l, \hat{e}_i \right\rangle_J \varepsilon_l\frac{\partial}{\partial a_{il}} \cdot A
 \\
 = \sum_{i,j=1}^m \left\langle \nabla_{\dot x(t)} e_j, e_i \right\rangle_J \varepsilon_iA \cdot \frac{\partial}{\partial a_{ij}}
 - \sum_{i,j=1}^m \left\langle\nabla_{q\dot x(t)} \hat{e}_j, \hat{e}_i \right\rangle_J \varepsilon_j\frac{\partial}{\partial a_{ij}} \cdot A
\\
= \sum_{i,j=1}^m \left\langle \nabla_{\dot x(t)} e_j, e_i \right\rangle_J \varepsilon_iA \cdot \frac{\partial}{\partial a_{ij}}
 -
\hspace*{-0.3 cm}\sum_{i,j=1}^m \left\langle\nabla_{q\dot x(t)} \hat{e}_j, \hat{e}_i \right\rangle_J \varepsilon_j
 \sum_{r,s=1}^m  a_{js}\varepsilon_r\varepsilon_i a_{ir} A \cdot \frac{\partial}{\partial a_{rs}}
\\
= \sum_{i,j=1}^m \left\langle \nabla_{\dot x(t)} e_j, e_i \right\rangle_J \varepsilon_iA \cdot \frac{\partial}{\partial a_{ij}}\\
 - \sum_{r,s = 1}^m \left\langle \nabla_{q\dot x(t)} (\sum_{j=1}^{m}\varepsilon_j a_{js}\hat{e}_j), (\sum_{i =1}^m \varepsilon_i a_{ir} \hat{e}_i) \right\rangle_J  \varepsilon_r A \cdot \frac{\partial}{\partial a_{rs}}
   \end{multline}
\begin{multline*}
=\sum_{i,j=1}^m \left\langle \nabla_{\dot x(t)} e_j, e_i \right\rangle_J \varepsilon_iA \cdot \frac{\partial}{\partial a_{ij}}
 - \sum_{r,s = 1}^m \left\langle \nabla_{q\dot x(t)} q e_s, q e_r \right\rangle_J\varepsilon_r  A \cdot \frac{\partial}{\partial a_{rs}}
 \\
 =\sum_{i,j=1}^m \left(\left\langle \nabla_{\dot x(t)} e_j, e_i \right\rangle_J  -  \left\langle \nabla_{q\dot x(t)} q e_j, q e_i \right\rangle_J\right)\varepsilon_i  A \cdot \frac{\partial}{\partial a_{ij}}.
\end{multline*}
Interchanging the indices $i$ and $j$ and noticing that the coefficients are skew symmetric, we get
\begin{equation}\label{sumdist1}
\sum_{i,j=1}^m \dot{a}_{ji} \frac{\partial}{\partial a_{ji}}
= \sum_{i,j=1}^m \left(\left\langle \nabla_{\dot x(t)} e_j, e_i \right\rangle_J - \left\langle \nabla_{q\dot x(t)} q e_j, q e_i \right\rangle_J\right)(-1)\varepsilon_j  A \cdot \frac{\partial}{\partial a_{ji}}.
\end{equation}
Summing~\eqref{sumdist} and~\eqref{sumdist1}  we deduce
\begin{equation} \label{A left}
\sum_{i,j=1}^m \dot{a}_{ij} \frac{\partial}{\partial a_{ij}}
= \sum_{i<j} \left(\left\langle \nabla_{\dot x(t)} e_j, e_i \right\rangle_J
 - \left\langle \nabla_{q\dot x(t)} q e_j, q e_i \right\rangle_J \right)\varepsilon_i A\cdot W_{ij}.
 \end{equation}
Written the same in a right invariant basis, we obtain
\begin{equation*}
\sum_{i,j=1}^m \dot{a}_{ij} \frac{\partial}{\partial a_{ij}}
= \sum_{i<j} \left(\left\langle \nabla_{\dot x(t)} q^{-1} \hat{e}_j, q^{-1} \hat{e}_i \right\rangle_J
 -  \left\langle \nabla_{q\dot x(t)} \hat{e}_j, \hat{e}_i \right\rangle_J \right) \varepsilon_i\big[\text{Ad}(A^{-1})\big] A\cdot W_{ij}.
 \end{equation*}

Similarly, (III) holds if and only if
\begin{multline}\label{Bleftright}
\sum\limits_{\kappa,\lambda=1}^\m \dot{b}_{\kappa\lambda} \frac{\partial}{\partial b_{\kappa\lambda}}
= \sum\limits_{\kappa <\lambda} \left(\left\langle \nabla_{\dot x(t)}^{\perp} \epsilon_\lambda, \epsilon_\kappa \right\rangle_J
 - \left\langle \nabla_{q\dot x(t)}^\perp p \epsilon_\lambda, p \epsilon_\kappa \right\rangle_J \right)\varepsilon_{\kappa} B\cdot W_{\kappa\lambda}. \\
= \sum\limits_{\kappa <\lambda} \left(
\left\langle \nabla_{\dot x(t)}^{\perp} p^{-1} \hat{\epsilon}_\lambda, p^{-1} \hat{\epsilon}_\kappa \right\rangle_J
 - \left\langle \nabla_{q\dot x(t)}^\perp \hat{\epsilon}_\lambda, \hat{\epsilon}_\kappa \right\rangle_J \right)\varepsilon_{\kappa}
\big[\text{Ad}(B^{-1})\big] B\cdot W_{\kappa\lambda}.
 \end{multline}

It may seem that all of the coefficients of $W_{ij}(A)$
in~\eqref{A left} vanish
from conditions (II). This is not true, however, due to  the subtle difference between the covariant derivative $\frac{D}{dt}$ along the curve $\widehat{x}(t)$ and the covariant derivative $\nabla_{\dot{\widehat{x}}(t)}$ along the vector fields $\dot{\widehat{x}}(t)$. Indeed, notice that
$$\frac{D}{dt} a_{sj}(t) \hat{e}_s(\widehat{x}(t)) = \dot{a}_{sj}(t) \hat{e}_s(\widehat{x}(t)) + a_{sj}(t) \frac{D}{dt} \hat{e}_s(\widehat{x}(t))
$$
and since $\{\hat{e}_s\}_{s=1}^{m}$ is extendable in a neighborhood of $\widehat{x}(t)$ we can continue and get
$$ =\dot{a}_{sj}(t) \hat{e}_s(\widehat{x}(t)) + a_{sj}(t) \nabla_{\dot{\widehat{x}}(t)}\hat{e}_s(\widehat{x}(t))
=\dot{a}_{sj} \hat{e}_s(\widehat{x}(t)) + a_{sj} \nabla_{q \dot{x}(t)}\hat{e}_s(\widehat{x}(t)).$$
While $\nabla_{\dot{\widehat x}(t)} a_{sj}(t) \hat{e}_s(x(t)) = a_{sj}(t) \nabla_{\dot{\widehat x}(t)} \hat{e}_s(x(t))$ due to the $\mathbb R$ linearity of the connection and the fact that the function $a_{sj}(t)$ depends on $t$ and is not defined as a function on $M$.
Similar relations hold for $\frac{D^{\perp}}{dt}$.

Observe that, due to the expressions~\eqref{A left} and~\eqref{Bleftright},  the vector field $\dot\gamma$ along $\gamma\in Q \oplus P_{\iota,\widehat{\iota}}$ can be considered as a "non-twisted lift" of the vector field $\dot x$ along the curve $x\in M$. We generalize this property on any local vector field on $M$.

\begin{definition}\label{defdistr} Non-twisted lifts of a vector field $X$ on $U\subset M$ are the vector fields
$\mathcal{V}(X)$ and $\mathcal{V}^{\perp}(X)$ on $\pi^{-1}(U\times\widehat U)\subset Q \oplus P_{\iota,\widehat{\iota}}$ satisfying
\begin{equation} \label{Vunderh}
dh \left(\mathcal{V}(X)(q,p) \right) = \sum_{i<j} \left(\left\langle \nabla_{X} e_j, e_i \right\rangle_J
 - \left\langle \nabla_{qX} q e_j, q e_i \right\rangle_J \right)\varepsilon_i A\cdot W_{ij}.
 \end{equation}
\begin{equation} \label{Vperpunderh}
dh \left(\mathcal{V}^{\perp}(X)(q,p)\right)
= \sum_{\kappa <\lambda} \left(\left\langle \nabla_{X}^{\perp} \epsilon_\lambda, \epsilon_\kappa \right\rangle_J
 - \left\langle \nabla_{qX}^\perp p \epsilon_\lambda, p \epsilon_\kappa \right\rangle_J \right) \varepsilon_{\kappa} B\cdot W_{\kappa\lambda}.
 \end{equation}
 for any local trivialization $h$ as in~\eqref{trivi} and any $(q,p) \in \pi^{-1}(U\times\widehat U)$.
 \end{definition}

Notice that since the covariant derivative along a vector field $X$ depends only on the value $X(x)$ at $x\in U\subset M$ we conclude that if $Y(x) = X(x) = v_x \in T_xM$, then the lifts $\mathcal{V}(Y)(q,p) = \mathcal{V}(X)(q,p)$
for every $(q,p) \in (Q \oplus P_{\iota,\widehat{\iota}})_{x\times\widehat x}$. Hence, we may define the lift
$\mathcal{V}(v_x)(q,p)$ for any vector $v_x \in T_xM$ and $(q,p) \in  (Q \oplus P_{\iota,\widehat{\iota}})_{x\times\widehat x}$.
The no-slip conditions imply that $qv_x\in T_{\widehat x}\widehat M$. Also notice that the map $X \mapsto \mathcal{V}(X)$ is linear. The same holds for $X \mapsto \mathcal{V}^{\perp}(X)$. This leads to the definition of the distributions contained in the following propositions.

\begin{proposition}
A curve $(q(t),p(t))$ in $Q \oplus P_{\iota,\widehat{\iota}}$ is a rolling if and only if it is a horizontal curve
with respect to the distribution $E$, defined by $$E_{(q,p)} = \left\{v_x + qv_x + \mathcal{V}(v_x)(q,p) + \mathcal{V}^\perp(v_x)(q,p) | \, v_x \in T_xM \right\},$$
where $(q,p) \in  (Q \oplus P_{\iota,\widehat{\iota}})_{x\times\widehat x}$.
\end{proposition}
\begin{proposition}\label{propdist}
A curve $q(t)$ in $Q$ is an intrinsic rolling if and only if it is a horizontal curve
with respect to the distribution $D$, defined by $$D_q = \left\{v_x + qv_x + \mathcal{V}(v_x)(q) | \, v_x \in T_xM \right\},\quad q \in Q_{x\times\widehat x}.$$
\end{proposition}


\section{Causal character of the rolling}\label{Causal character}

The specific feature of pseudo-Riemannian manifolds is the causal structure, or division of all vectors into three classes timelike, spacelike and nullike (or lightlike for the metric of index one). It is easy to see the following

\begin{proposition}
If a rolling curve $x\colon I\to M$ is of one of the causal types, then the development curve $\hat x$ is of the same type.
\end{proposition}

\begin{proof} Since the map $q\colon T_{x(t)}M\to T_{\hat x(t)}\widehat M$ is an isometry then the no-slip condition implies
$
\langle\dot{\hat x}(t),\dot{\hat x}(t)\rangle_J = \langle q(t)\dot{ x}(t), q(t)\dot{x}(t)\rangle_J =\langle\dot{x}(t),\dot{x}(t)\rangle_J.
$
\end{proof}

The pseudo-orthogonal group also admits the scalar product as was mentioned in Subsection~\ref{causality} that we denoted by $\langle\langle .\,,.\rangle\rangle_J$. Under the local trivialization $h$ as in Subsection~\ref{local triv} a rolling curve takes the form $h(q(t),p(t))=\gamma(t)=\big(x(t),\hat x(t)), A(t),B(t)\big)$, $t\in I$.
We know that the curves $x$ and $\hat x$ have the same causal character. We ask whether the curves $A\in G_{\mu}(m)$ and $B\in G_{\nu-\mu}(\m)$ have the same causal character? As we saw for the benchmark example, the Lorentzian sphere, and for the symmetric spaces it is true under the classical rolling. In the following theorem we give a partial answer to this question in general case.
\begin{theorem}
If $\gamma(t)=\big(x(t),\hat x(t), A(t),B(t)\big)$ is a rolling curve under the local trivialization then the causal character of curves $A(t)$, $B(t)$ can be calculated as follows.
The curve $A$ is timelike $($spacelike or null$)$ if the expression
$$
\sum_{i,h=1}^{m}\varepsilon_i\varepsilon_h\Big(\sum_{k=1}^{m}\big[\dot x^k\Gamma^{i}_{kh}-c_{ih}\dot{\widehat x^k}\widehat\Gamma^{i}_{kh}\big]\Big)^2,
$$ is negative $($positive or zero$)$, respectively.
Here $c_{ih}=\sum_{r,s=1}^{m}\varepsilon_r\varepsilon_sa_{rh}a_{si}$.
The curve $B$ is timelike $($spacelike or null$)$ if the expression
$$\sum_{\kappa,\chi=1}^{\m}\varepsilon_{\kappa}\varepsilon_{\chi}\Big(\sum_{l=1}^{\m}\big[\dot x^{l}\big(\Gamma^{\perp}\big)^{\kappa}_{l\chi}-d_{\kappa \chi}\dot{\widehat x^l}\big(\widehat\Gamma^{\perp}\big)^{\kappa}_{l\chi}\big]\Big)^2
$$
is negative $($positive or zero$)$, respectively. Here $d_{\kappa\chi}=\sum_{\rho,\sigma=1}^{\m}\varepsilon_{\rho}\varepsilon_{\sigma}b_{\rho\chi}b_{\sigma\kappa}$.
\end{theorem}
\begin{proof}
We start from the general observation. If $A\colon I\to G_{\mu}(m)$ is a curve then $\dot A(t)=A(t)\cdot U(t)$, where $U$ is a curve in the Lie algebra of $G_{\mu}(m)$. Then
$$
\langle\langle \dot A(t),\dot A(t)\rangle\rangle_J=\tr(\dot A^J\dot A)=-\tr U^2,
$$
since $A^{\mathbf t}JA=J$ and $JU^{\mathbf t}J=-U$.

Under the local trivialization the derivative $\dot{\gamma}$ was presented as
$$
\dot{\gamma}(t) = \dot{x}(t) + \dot{\widehat{x}}(t) + \sum_{i,j = 1}^m \dot{a}_{ij} \frac{\partial}{\partial a_{ij}}
+ \sum_{\kappa,\lambda = 1}^{\m} \dot{b}_{\kappa\lambda} \frac{\partial}{\partial b_{\kappa\lambda}}=\dot{x}(t) + \dot{\widehat{x}}(t)+\dot A(t)+\dot B(t).
$$
We start from study of $\dot A$. According to~\eqref{A left} we have
$\dot A= A\cdot U$ with
$$U=\sum_{i<j} \left(\left\langle \nabla_{\dot x(t)} e_j, e_i \right\rangle_J
 - \left\langle \nabla_{q\dot x(t)} q e_j, q e_i \right\rangle_J \right)\varepsilon_i W_{ij}(1).
 $$
 We denote the coefficients of $\varepsilon_iW_{ij}(1)$ by $$w_{ij}=\left(\left\langle \nabla_{\dot x(t)} e_j, e_i \right\rangle_J
 - \left\langle \nabla_{q\dot x(t)} q e_j, q e_i \right\rangle_J \right).$$
 Observe $\langle \nabla_{\dot x(t)} e_j, e_i \rangle_J=\sum_{k=1}^{m}\dot x^k(t)\langle \nabla_{e_k} e_j, e_i \rangle_J=\sum_{k=1}^{m}\dot x^k(t)\Gamma_{kj}^i(x(t))$ and
 \begin{align*}
\langle \nabla_{q\dot x(t)} q e_j, q e_i \rangle_J & = \sum_{l=1}^{m}\dot {\widehat{x}^l}(t)\langle \nabla_{\widehat e_l} \sum_{r=1}^{m}\varepsilon_ra_{rj} \widehat e_j, \sum_{s=1}^{m}\varepsilon_sa_{si} \widehat e_i \rangle_J
\\
&=\sum_{s,r=1}^{m}\varepsilon_r\varepsilon_sa_{rj}a_{si}\sum_{l=1}^{m}\dot {\widehat{x}^l}(t)\widehat\Gamma_{lj}^{i}(x(t)),
\end{align*}
where $\Gamma_{kj}^i(x(t))$ and $\widehat\Gamma_{lj}^{i}(x(t))$ are Christoffel symbols of Levi-Civita connections for $M$ and $\hat M$ along curves $x$ and $\widehat x$, respectively.

Since for the trace we need only information about the diagonal terms of $U^2$ we find
$$\{U^2\}_{ii}=\sum_{h=1}^{m}\{U\}_{ih}\{U\}_{hi}=\sum_{h=1}^{m}\varepsilon_hw^2_{ih}\quad\text{for}\quad i=1,\ldots \mu.
$$
$$
\{U^2\}_{ii}=\sum_{h=1}^{m}\{U\}_{ih}\{U\}_{hi}=\sum_{h=1}^{m}-\varepsilon_hw^2_{ih}\quad\text{for}\quad i=\mu,\ldots m.
$$
Thus the trace is expressed as followed
$$
-\tr U^2=-\sum_{i,h=1}^{m}-\varepsilon_i\varepsilon_hw_{ih}^2=\sum_{i,h=1}^{m}\varepsilon_i\varepsilon_h\Big(\sum_{k=1}^{m}\big[\dot x^k\Gamma^{i}_{kh}-c_{ih}\dot{\widehat x^k}\widehat\Gamma^{i}_{kh}\big]\Big)^2,
$$
where $c_{ih}=\sum_{r,s=1}^{m}\varepsilon_r\varepsilon_sa_{rh}a_{si}$.

Analogously for $\dot B(t)=B(t)\cdot \mathcal U(t)$, where $B=\{b_{\kappa\lambda}\}_{\kappa\lambda=1}^{\m}$ is a curve in the group $G_{\nu-\mu}(\m)$ and $\mathcal U$ is a curve in the Lie algebra of $G_{\nu-\mu}(\m)$ we have
$$
-\tr \mathcal U^2=-\sum_{\kappa,\chi=1}^{\m}-\varepsilon_{\kappa}\varepsilon_{\chi}w_{\kappa \chi}^2=\sum_{\kappa,\chi=1}^{\m}\varepsilon_{\kappa}\varepsilon_{\chi}\Big(\sum_{l=1}^{\m}\big[\dot x^{l}\big(\Gamma^{\perp}\big)^{\kappa}_{l\chi}-d_{\kappa \chi}\dot{\widehat x^l}\big(\widehat\Gamma^{\perp}\big)^{\kappa}_{l\chi}\big]\Big)^2,
$$
where $d_{\kappa\chi}=\sum_{\rho,\sigma=1}^{\m}\varepsilon_{\rho}\varepsilon_{\sigma}b_{\rho\chi}b_{\sigma\kappa}$.
\end{proof}



\section{Extended configuration space}\label{extended}


In the present section we would like to describe the embedding of the configuration spaces $Q$ and $Q\oplus P_{i,\widehat i}$ into, so called, extended configuration spaces. One of main difficulties to work with $Q$ and $Q\oplus P_{i,\widehat i}$ is that these bundles are not principal bundles, they are just a fiber bundles whose typical fiber under the local trivialization is diffeomorphic to one of the groups $G_{\mu}(m)$ or $G_{\nu-\mu}(\m)$. The extended configuration spaces are the vector bundles and the fiber bundles $Q$ and $Q\oplus P_{i,\widehat i}$ form subbundles of them. This idea was quite successfully exploit in~\cite{ChKok}. We give necessary definitions.

We start from the configuration space $Q$. It is well known that the space $\Hom(V,W)$ of linear maps between two real vector spaces can be identified with the tensor product $V^*\otimes W$, where $V^*$ is the dual to $V$. Applying this to the vector spaces $V=T_xM$ and $W=T_{\widehat x}\widehat M$ we obtain the tensor product $T_x^*M\otimes T_{\widehat x}\widehat M$. Since we are interested in finding the configuration space over the product $M\times\widehat M$, we use the coordinate independent embeddings
$$
T^*_xM\subset T^*_{(x,\widehat x)}(M\times \widehat M)\cong T^*_xM\times T^*_{\widehat x}\widehat M\quad\text{ and}
$$
$$
T_{\widehat x}\widehat M\subset T_{(x,\widehat x)}(M\times \widehat M)\cong T_xM\times T_{\widehat x}\widehat M.
$$
 Therefore, the space $T_xM\otimes T_{\widehat x}\widehat M$ can be canonically included into the space
$$\mathcal T^1_1(M\times\widehat M)_{(x,\widehat x)}:=T^*_{(x,\widehat x)}(M\times \widehat M)\otimes T_{(x,\widehat x)}(M\times \widehat M)$$ of $(1,1)$-tensors at point $(x,\widehat x)\in M\times \widehat M$. Taking the disjoin union over $(x,\widehat x)\in M\times \widehat M$, we can consider $T^*M\otimes T\widehat M$ as a vector subbundle
\begin{equation}\label{Pi}
\Pi\colon T^*M\otimes T\widehat M\to M\times \widehat M
\end{equation}
of a tensor bundle $\mathcal T^1_1(M\times\widehat M)$. We claim that the bundle $\pi_Q\colon Q\to M\times \widehat M$ is a subbundle of~\eqref{Pi}. Indeed since the manifolds $M$ and $\widehat M$ are endowed with the metric, the configuration space $Q$ is defined as a following subset of $T^*M\otimes T\widehat M$:
$$
\begin{array}{ll}
Q_{(x,\widehat x)} =  \Big\{&  q\in(T^*M\otimes T\widehat M)_{(x,\widehat x)}\mid  (x,\widehat x)\in M\times \widehat M,\ q\text{ is an isometry }
\\
\smallskip
&   \text{preserving the chosen (space, time or space-time) orientation}\Big\}.
\end{array}
$$
Recall that there is a diffeomorphism $h_Q$ defining the trivialization
\begin{equation} \label{triviQ}
\begin{array}{rcl}
Q \supset \pi^{-1}_Q(U \times \widehat{U}) & \stackrel{h_Q}{\to} & U \times \widehat{U} \times G_{\mu}(m)
\\
q & \mapsto & (x,\widehat{x}, A).
\end{array}
\end{equation}
The local trivialization $h_Q$ of the bundle $\pi_Q\colon Q\to M\times \widehat M$ can be considered as a restriction of the local trivialization
$$
\Pi^{-1}(U \times \widehat{U}) \longrightarrow U \times \widehat{U}\times\mathfrak{gl}(m).
$$
If the metric is positive definite, then $G_{\mu}(m) $ is simply the group $\SO(m)$. The same arguments as in~\cite{ChKok,GGLM2} shows that $Q$ is a smooth subbundle of $T^*M\otimes T\widehat M$, and that the bundle $Q$ is not a $G_{\mu}(m)$-principle bundle in the case $m>2$.

Analogously, given the isometric embeddings $\iota\colon M\to\mathbb R^{m+\m}_{\nu}$ and $\widehat \iota\colon \widehat M\to\mathbb R^{m+\m}_{\nu}$ we define the vector bundle
\begin{equation}\label{PiPerp}
\Pi^{\perp}\colon T^{\perp*}M\otimes T^{\perp}\widehat M\to M\times \widehat M.
\end{equation}
Then the disjoint union $P_{\iota,\widehat\iota}$ of sets of all orientation preserving isometries $p\colon T^{\perp}_xM \to T^{\perp}_{\widehat x}\widehat M$ becomes the smooth subbundle of~\eqref{PiPerp}. The trivialization
\begin{equation} \label{triviP}
\begin{array}{rcl}
P_{\iota,\widehat\iota} \supset \pi^{-1}_{P_{\iota,\widehat\iota}}(U \times \widehat{U}) & \stackrel{h_Q}{\to} & U \times \widehat{U} \times G_{\nu-\mu}(\m)
\\
p & \mapsto & (x,\widehat{x}, B).
\end{array}
\end{equation}
can be also considered as a restriction of the local trivialization
$$
(\Pi^{\perp})^{-1}(U \times \widehat{U}) \longrightarrow U \times \widehat{U}\times\mathfrak{gl}(\m).
$$
We conclude that the fiber bundle $Q\oplus P_{\iota,\widehat\iota}$ is a smooth subbundle of the vector bundle
$$
\Pi\oplus\Pi^{\perp}\colon \Big(T^*M\otimes T\widehat M\Big)\oplus\Big(T^{\perp*}M\otimes T^{\perp}\widehat M\Big)\to M\times \widehat M
$$
and the trivialization~\eqref{trivi}
\begin{equation*}
\begin{array}{rcl}
Q\oplus P_{\iota,\widehat\iota} \supset \pi^{-1}(U \times \widehat{U}) & \stackrel{h}{\to} & U \times \widehat{U} \times G_{\mu}(m)\times G_{\nu-\mu}(\m)
\\
(q,p) & \mapsto & (x,\widehat{x}, A,B).
\end{array}
\end{equation*}
is the restriction of the trivialization
$$
(\Pi\oplus \Pi^{\perp})^{-1}(U \times \widehat{U}) \longrightarrow U \times \widehat{U}\times\mathfrak{gl}(m)\times \mathfrak{gl}(\m).
$$


\section{Appendix - The tangent space of $G_{\mu}(m)$}\label{tangent SOn}

We describe the tangent space  $TG_{\mu}(m)$ in terms of left and right invariant vector fields. Following the notation of Subsection~\ref{local triv}, we use the isomorphism $h$ to identify the tangent spaces under trivialization:
$$T\pi^{-1}(U \times \widehat{U}) \cong TU \times T\widehat{U} \times TG(\mathbb R^{m}_{\mu}) \times TG(\mathbb R^{\m}_{\nu-\mu}).$$

The tangent space at the identity of $G_{\mu}(m)$, or the Lie algebra $\mathfrak g_{\mu}(m)$,  is spanned by the skew symmetric part
$$W_{ij}= \frac{\partial}{\partial a_{ij}} - \frac{\partial}{\partial a_{ji}}, \quad\text{if}\quad 1 \leq i < j \leq \mu,\quad\text{or}\quad \mu+1 \leq i < j \leq m,$$
and the symmetric part
$$
W_{ij}= \frac{\partial}{\partial a_{ij}} + \frac{\partial}{\partial a_{ji}}, \quad\text{if}\quad 1 \leq i\leq\mu < j \leq m.
$$
We write the basis $W_{ij}$ in the homogeneous form by making use of the sign symbol given by the scalar product
$$
\langle e_i,e_j\rangle_J=\langle \hat e_i,\hat e_j\rangle_J=\varepsilon_i\delta_{ij},\quad
\varepsilon_i=
\begin{cases}
-1 &\quad\text{if}\quad 1\leq i\leq\mu,
\\
1 &\quad\text{if}\quad \mu+1\leq i\leq m,
\end{cases}
$$
where $\delta_{ij}$ is the Kronecker symbol. Thus
\begin{equation*}
\{ W_{ij}= \frac{\partial}{\partial a_{ij}} - \varepsilon_{i}\varepsilon_{j}\frac{\partial}{\partial a_{ji}}, \quad 1 \leq i < j \leq m\}
\end{equation*}
generates the tangent space of $G(\mathbb R^{m}_{\mu})$ at the identity. If we write~\eqref{basis} in the form $\varepsilon_i W_{ij}=\varepsilon_i\frac{\partial }{\partial a_{ij}}-\varepsilon_j\frac{\partial }{\partial a_{ji}}$, then we observe the property $\varepsilon_i W_{ij}=-\varepsilon_j W_{ji}$.

Since the left and right action of $G_{\mu}(m)$ on the tangent space $TG_{\mu}(m)$ is described by
$$A \cdot \frac{\partial}{\partial a_{ij}} = \sum_{r=1}^m a_{ri} \frac{\partial}{\partial a_{rj}}, \qquad
\frac{\partial}{\partial a_{ij}} \cdot A = \sum_{s =1}^m a_{js} \frac{\partial}{\partial a_{is}}.$$
then, the left and right translations by $A\in G_{\mu}(m)$ of the basis elements in (\ref{basis}) defines vectors
\begin{equation*}
A \cdot W_{ij}(1) = \sum_{r =1}^m \left(a_{ri} \frac{\partial}{\partial a_{rj}} - \varepsilon_{i}\varepsilon_{j}a_{rj} \frac{\partial}{\partial a_{ri}} \right)
\end{equation*}
as a global left invariant basis of $TG_{\mu}(m)$
and
\begin{equation*} 
W_{ij}(1) \cdot A= \sum_{r =1}^m \left(a_{js} \frac{\partial}{\partial a_{is}} - \varepsilon_{i}\varepsilon_{j}a_{is} \frac{\partial}{\partial a_{js}} \right)
\end{equation*}
as a global right invariant basis of $TG_{\mu}(m)$.

We want to present the formula expressing the left invariant basis $A \cdot W_{ij}(1)$ in terms of the right invariant basis $W_{ij}(1) \cdot A$ and vice versa.
Recall the notation $A^J=JA^tJ$, and observe that the multiplication from the left by $J=\diag(I_{\mu},I_{\nu-\mu})$ changes the sign of the first $\mu$ rows and the multiplication from the right by $J$ change the sign of the first $\nu-\mu$ columns. Therefore, for $A=\{a_{ij}\}$, we have $A^J=\{a_{ij}^J\}=\{\varepsilon_{i}\varepsilon_{j}a_{ji}\}$. Then
$$
\{A^JA\}_{jl}=\sum_{s=1}^{m}a^J_{js}a_{sl}=\sum_{s=1}^{m}\varepsilon_{j}\varepsilon_{s}a_{sj}a_{sl}=\delta_{lj},$$
or
$$
\{AA^J\}_{jl}=\sum_{s=1}^{m}a_{js}a_{sl}^J=\sum_{s=1}^{m}\varepsilon_{l}\varepsilon_{s}a_{js}a_{ls}=\delta_{lj}.
$$
Thus we obtain the following formula to switch from left to right translation
$$A \cdot \frac{\partial}{\partial a_{ij}} = \sum_{r=1}^m a_{ri} \frac{\partial}{\partial a_{rj}}
= \sum_{l, r =1}^m a_{ri} \delta_{jl} \frac{\partial}{\partial a_{rl}}
= \sum_{l, r, s =1}^m \varepsilon_{l}\varepsilon_{s}a_{ri} a_{sj} a_{sl} \frac{\partial}{\partial a_{rl}}$$
$$= \sum_{r,s=1}^m \varepsilon_{l}\varepsilon_{s}a_{ri} a_{sj} \left(\frac{\partial}{\partial a_{rs}} \cdot A \right),$$
and the other way around,
$$\frac{\partial}{\partial a_{ij}} \cdot A = \sum_{s =1}^m a_{js} \frac{\partial}{\partial a_{is}}
= \sum_{l, s =1}^m a_{js} \delta_{li} \frac{\partial}{\partial a_{ls}}
= \sum_{l, r, s =1}^m \varepsilon_{r}\varepsilon_{i}a_{js} a_{lr} a_{ir} \frac{\partial}{\partial a_{ls}}$$
$$= \sum_{r, s =1}^m \varepsilon_{r}\varepsilon_{i}a_{js} a_{ir} \left(A \cdot \frac{\partial}{\partial a_{rs}}\right).$$

Moreover
\begin{equation}\label{first}
\Big(\frac{\partial}{\partial a_{ij}}-\varepsilon_{i}\varepsilon_{j}\frac{\partial}{\partial a_{ji}}\Big)\cdot A =
\sum_{r,s=1}^{m} \varepsilon_{i}\varepsilon_{r}\big(a_{js} a_{ir} -a_{is} a_{jr})\big) A\cdot\frac{\partial}{\partial a_{rs}}
\end{equation}
and from other side interchanging $r$ and $s$ we obtain
\begin{eqnarray}\label{second}
\Big(\frac{\partial}{\partial a_{ij}}-\varepsilon_{i}\varepsilon_{j} \frac{\partial}{\partial a_{ji}}\Big)\cdot A=
\sum_{r,s=1}^{m} \varepsilon_{s}\varepsilon_{i}\big(a_{jr} a_{is} -a_{ir} a_{js})\big) A\cdot\frac{\partial}{\partial a_{sr}}.
\end{eqnarray}
Summing~\eqref{first} and~\eqref{second} and observing that $\varepsilon_iW_{ij}=-\varepsilon_jW_{ji}$, we get for $i<j$
$$
W_{ij}(1) \cdot A =\sum_{r<s} \varepsilon_{i}\varepsilon_{r}\big(a_{js} a_{ir} -a_{is} a_{jr}\big)A\cdot W_{rs}(1).
$$
We also notice that
$$W_{ij}(1) \cdot A = \mathrm{Ad}(A^{-1}) W_{ij} (A),\quad\text{with}\quad A^{-1}=A^J.$$

Now we shall calculate the commutators of $W_{ij}$ based on formula
$$
[\frac{\partial}{\partial a_{ij}},\frac{\partial}{\partial a_{kl}}]=\delta_{jk}\frac{\partial}{\partial a_{il}}-\delta_{il}\frac{\partial}{\partial a_{kj}},
$$
to obtain
$$
\begin{array}{lll}
[W_{ij}, W_{kl}]  & = \delta_{jk} (\frac{\partial}{\partial a_{il}}-\varepsilon_{i}\varepsilon_j\varepsilon_{k}\varepsilon_{l}\frac{\partial}{\partial a_{li}}) - \delta_{il} (\frac{\partial}{\partial a_{kj}}-\varepsilon_{i}\varepsilon_{j}\varepsilon_{k}\varepsilon_{l}\frac{\partial}{\partial a_{jk}})
\\
&
+ \delta_{ik}(-\varepsilon_{i}\varepsilon_{j}\frac{\partial}{\partial a_{jl}}+\varepsilon_{k}\varepsilon_{l}\frac{\partial}{\partial a_{lj}}) - \delta_{jl} (-\varepsilon_{i}\varepsilon_{j}\frac{\partial}{\partial a_{ki}}+\varepsilon_{k}\varepsilon_{l}\frac{\partial}{\partial a_{ik}}).
\end{array}
$$
Observe that, if $\varepsilon_{i}\varepsilon_{j}=\varepsilon_{k}\varepsilon_{l}=\pm1$, the commutator is a skew-symmetric matrix, in the case $\varepsilon_{i}\varepsilon_{j}=-\varepsilon_{k}\varepsilon_{l}$ one obtains a symmetric matrix.

Each basis vector $\frac{\partial}{\partial a_{ij}}$ can be written in the matrix form by using the standard notation of $(m\times m)$-matrices $E_{ij}$ with zero entries except of 1 at the $i$-row and $j$-column. Then
$$
W_{ij}(1)= E_{ij} -\varepsilon_{i}\varepsilon_{j}E_{ji}, \quad\text{if}\quad 1 \leq i < j \leq m.
$$
and all actions are written as a matrix multiplication
$$
A\cdot W_{ij}(1)=A(E_{ij} -\varepsilon_{i}\varepsilon_{j}E_{ji})\quad\text{and}\quad W_{ij}(1)\cdot A=(E_{ij}  -\varepsilon_{i}\varepsilon_{j}E_{ji})A.
$$
The commutation relations are written as
$$
\begin{array}{lll}
[W_{ij}, W_{kl}]  & = \delta_{jk} (E_{il}-\varepsilon_{i}\varepsilon_j\varepsilon_{k}\varepsilon_{l}E_{li}) - \delta_{il} (E_{kj}-\varepsilon_{i}\varepsilon_{j}\varepsilon_{k}\varepsilon_{l}E_{jk})
\\
&
+ \delta_{ik}(-\varepsilon_{i}\varepsilon_{j}E_{jl}+\varepsilon_{k}\varepsilon_{l}E_{lj}) - \delta_{jl} (-\varepsilon_{i}\varepsilon_{j}E_{ki}+\varepsilon_{k}\varepsilon_{l}E_{ik}).
\end{array}
$$


\section{Acknowledgement}

This work was developed whilst the second author visited the
University of Bergen in 2012. 




\end{document}